\documentclass[reqno]{amsart}
\usepackage{amsmath, amsthm, amssymb, amstext}

\usepackage{hyperref,xcolor}
\hypersetup{
 pdfborder={0 0 0},
 colorlinks,
}
\usepackage[left=2.5cm,right=2.5cm,top=3cm,bottom=3cm]{geometry}

\usepackage[colorinlistoftodos]{todonotes}
\usepackage{enumitem}
\newtheorem{theorem}{Theorem}[section]
\newtheorem{lemma}[theorem]{Lemma}
\newtheorem{e-proposition}[theorem]{Proposition}

\newtheorem{e-definition}[theorem]{Definition\rm}
\newtheorem{remark}[theorem]{\it Remark\/}

\def\div{{\rm div}}

\DeclareMathOperator*{\divergenz}{div}              %

\newcommand{\N}{\mathbb{N}}

\newcommand{\R}{\mathbb{R}}

\newcommand{\bw}{\mathbf{w}}
\newcommand{\bz}{\mathbf{z}}
\newcommand{\bu}{\mathbf{u}}
\newcommand{\bv}{\mathbf{v}}
\newcommand{\bphi}{\boldsymbol{\phi}}
\newcommand{\bpsi}{\boldsymbol{\psi}}
\newcommand{\bta}{\boldsymbol{\eta}}

\newcommand{\eps}{\varepsilon}

\newcommand{\Om}{\Omega}

\newcommand{\into}{\int_{\Omega}}

\renewcommand{\l}{\left}
\renewcommand{\r}{\right}
\numberwithin{theorem}{section}
\numberwithin{equation}{section}

\title[Existence of weak solutions to a cross-diffusion Cahn-Hilliard type system]{Existence of weak solutions to  a cross-diffusion Cahn-Hilliard type system}

\author[V.\,Ehrlacher]{Virginie Ehrlacher}
\address[V.\,Ehrlacher]{CERMICS, Ecole des Ponts, INRIA, MATHERIALS team-project, 6 et 8 av. Blaise Pascal, Cit\'e Descartes 77455 Marne-La-Vall\'ee, France}
\email{virginie.ehrlacher@enpc.fr}

\author[G.\,Marino]{Greta Marino}
\address[G.\,Marino]{Technische Universit\"{a}t Chemnitz, Fakult\"at f\"ur Mathematik, Reichenhainer Stra\ss e 41, 09126 Chemnitz, Germany}
\email{greta.marino@mathematik.tu-chemnitz.de}

\author[J.-F.\,Pietschmann]{Jan-Frederik Pietschmann}
\address[J.-F.\,Pietschmann]{Technische Universit\"{a}t Chemnitz, Fakult\"at f\"ur Mathematik, Reichenhainer Stra\ss e 41, 09126 Chemnitz, Germany}
\email{jfpietschmann@math.tu-chemnitz.de}



\begin{document}

\begin{abstract}
The aim of this article is to study a Cahn-Hilliard model for a multicomponent mixture with cross-diffusion effects, degenerate mobility and where only one of the species does separate from the others. 
We define a notion of weak solution adapted to possible degeneracies and our main result is (global in time) existence. 
In order to overcome the lack of a-priori estimates, our proof uses the formal gradient flow structure of the system and an extension of the boundedness by entropy method which involves a careful analysis of an auxiliary variational problem. 
This allows to obtain solutions to an approximate, time-discrete system. Letting the time step size go to zero, we recover the desired weak solution where, due to their low regularity, 
the Cahn-Hilliard terms require a special treatment.
\end{abstract}

\maketitle


\section{Introduction}

The aim of this work is to study a Cahn-Hilliard model with degenerate mobility for a multicomponent mixture where cross-diffusion effects between the different species of the system are taken into account, 
and where only one species does separate from the others. The motivation for considering such a model stems from the fact that there exist multiphase systems where miscible entities may coexist in one single phase of the system, \cite{klinkert2015comprehension,Wenisch2016}. In the latter phase, cross-diffusion effects between the different miscible chemical species may have to be taken into account in order to correctly 
model the evolution of the concentrations or of the volumic fractions of each species. 

\medskip

More precisely, let $\Omega$ be a regular open bounded subdomain of $\mathbb{R}^d$ with $d = 1,2,3$.
We assume that the mixture is composed of $n+1$ species for some positive $n \in \N\setminus \{0\}$, occupying the domain $\Omega$. Let $T>0$ be some final time. For all $0\leq i \leq n$, we denote by $u_i(t,x)$ the volumic fraction of the $i^{th}$ species at time $t\in [0,T]$ and point $x\in \Omega$. 
We are interested in proving the existence of weak solutions $ \bu:= (u_0, u_1, \dots, u_n)$ to a system of the form:
\begin{equation}\label{system}
 \begin{split}
  \partial_t  \bu&=\div\left( M(\bu) \nabla  \mu\right), \\
 \end{split}
\end{equation}
that satisfy 
$$
\forall \, 0\leq i \leq n, \;  0\leq u_i(t,x) \leq 1,\quad \text{ and }\quad \sum_{i=0}^n u_i(t,x) = 1\text{ 
for a.e. }t\in [0,T],\;x\in \Omega.
$$
Here, for all $\bu \in \mathbb{R}_+^{n+1}$, $M(\bu)\in \mathbb{R}^{(n+1) \times (n+1)}$ is a degenerate mobility matrix whose precise expression is given in Section~\ref{sect2}, while $\mu$ is the chemical potential, defined as 
\begin{equation*} 
\mu = D_{\bu} E(\bu).
\end{equation*}
In this work, the energy functional $E: L^\infty(\Omega)^{n+1} \to \mathbb{R}\cup\{\pm \infty\}$ is given by
\begin{equation}
\label{energy}
E(\bu):= \left\{
\begin{array}{ll} 
 \int_\Omega \sum_{i=0}^n (u_i \ln u_i -u_i +1) + \frac{\eps}{2} |\nabla u_0|^2 + \beta u_0 (1-u_0) dx, & \mbox{if }u_0 \in H^1(\Omega),\\
 +\infty, & \mbox{otherwise},\\
\end{array}
\right.
\end{equation}
for some constants $\eps>0$ and $\beta>0$. The logarithmic terms in this energy functional account for diffusion while the other two terms are responsible for phase separation: the gradient term penalizes transitions 
while the last term encourages $u_0$ to be either one or zero. Note that in contrast to most other multi-phase Cahn-Hilliard systems, in our case, the last two terms only act on $u_0$, 
a situation called \emph{degenerate Ginsburg-Landau energy} in  \cite{Garcke2018_multiphasetumour}. The model is equivalent to the standard classical two-phase Cahn-Hilliard model with degenerate mobility in the 
case where $n=1$ and $\beta = \frac{1}{\eps}$ (see \cite{EG} for example). In the case when $\eps = \beta = 0$, system (\ref{system}) boils down to a
multi-species degenerate cross-diffusion system with size exclusion that was studied in the following series of publications~\cite{BDPS2010,BE2018freeboundary,Berendsen2019,cances2020convergent}. Let us also mention 
\cite{Berendsen2017nonlocal} where the cross-diffusion system with non-local interactions was studied which can be seen as a non-local precursor of our system.

Let us put our work into perspective with respect to previous results for multi-species Cahn-Hilliard and cross-diffusion systems. 

\subsubsection*{Cahn-Hilliard systems}
The scalar Cahn-Hilliard equation was introduced in \cite{CahnHilliard1958} as a model for phase separation. 
Existence of weak solution was first shown in the case of constant mobility, see e.g. \cite{Elliott1986,Caffarelli1995}, and later extended to degenerate, 
concentration dependent mobilities \cite{EG}. For more details we refer the reader to the review \cite{novick2008cahn} and the monograph \cite{Miranville2019}.
Multi-species Cahn-Hilliard systems have been studied in several earlier works and usually consider an energy functional of the form 
\begin{align*} 
E({\bf u}):= \int_\Omega\left[ \Psi({\bf u}) + \frac{1}{2}\nabla {\bf u} \cdot \Gamma \nabla {\bf u}\right] dx,
\end{align*}
for some symmetric semi-definite positive matrix $\Gamma \in \mathbb{R}^{(n+1) \times (n+1)}$ and bulk free-energy functional $\Psi$. 
In~\cite{elliott1991generalized}, Elliott and Luckhaus proved a global existence result for such a multiphase Cahn-Hilliard system with constant mobility and 
$\Gamma = \gamma {\rm I}$ for some $\gamma >0$. In~\cite{Elliot1997}, the authors generalized their result to the case of a degenerate concentration-dependent mobility matrix with a positive definite matrix $\Gamma$. 
Recently, in~\cite{boyer2014hierarchy}, the authors proposed a novel hierarchy of multi-species Cahn-Hilliard systems which are consistent with the standard two-species Cahn-Hilliard system, and which read as the model introduced above with 
$\Gamma$ a positive-definite matrix, a constant mobility, and a particular bulk energy functional $\Psi$. Numerical methods for such systems were proposed and analyzed in several contributions, see e.g.~\cite{eyre1998unconditionally,Wu2017}.

In all these works, global existence results are obtained for various mobility matrices $M$, bulk energy functionals $\Psi$, and, at least up to our knowledge, 
always for a positive definite matrix $\Gamma$ which implies, from a modeling point of view, that each species composing the mixture has the tendency to separate from all other species. 
This also facilitates the analysis compared to our system since it ensures $H^2$-regularity in space for all species $u_i$, $i=0,\ldots, n$, while this can only be expected for $u_0$ in our case.

\medskip

\subsubsection*{Cross-diffusion systems with size exclusion} 

Systems of partial differential equations with cross-diffusion have gained a lot of interest in recent years \cite{Kfner1996InvariantRF,Chen2004,Chen2006,Lepoutre2012,Juengel2015boundedness} 
and appear in many applications, for instance the modeling of population dynamics of multiple species~\cite{Burger2016} or cell sorting or chemotaxis-like applications~\cite{Painter2002,Painter2009}.

One major difficulty in the analysis of such strongly coupled systems is the lack of a priori estimates. Maximum principles are 
not available in general and since such systems are often only degenerate parabolic, classical energy estimates obtained by (formally) testing with the solution itself do not work. 
In particular it is not possible to obtain $L^\infty$ bounds (e.g. non-negativity) by choosing suitable test functions as done in \cite{Elliot1997} for a multi-species Cahn-Hilliard system.
%
For some cross-diffusion systems that feature an \emph{entropic} or \emph{formal gradient flow} structure, these issues can be overcome. More precisely, for systems that can be written as 
$$
\partial_t \bu = \div(M(\bu) \nabla  \partial_\bu e(\bu)),
$$
where $e : \mathcal Z \to \R$ is the \emph{entropy density} corresponding to the \emph{entropy functional}
$$
E(\bu) = \int_\Omega e(\bu)\;dx, 
$$
with
\begin{equation*} 
\mathcal Z:= \left\{ \bu:=(u_0, \dots, u_n) \in \mathbb{R}_+^{n+1}, \quad \sum_{i=0}^n u_i  = 1 \right\},
\end{equation*}
and for all $\bu:=(u_0, \dots, u_n) \in \mathcal Z$, $\partial_u e(\bu) = (\partial_{u_0} e(\bu), \ldots, \partial_{u_n} e(\bu))$. 
If the mobility matrix is positive semi-definite, a formal calculation immediately shows that the entropy is non-increasing since 
$$
\frac{d}{dt} E(\bu) = - \int_\Omega  \nabla \partial_\bu e(\bu)^t M(\bu) \nabla \partial_\bu e(\bu)\;dx \le 0.
$$
Thus all quantities appearing in the entropy remain bounded if the entropy of the initial configuration is finite. The lack of maximum principles can be compensated by introducing 
entropy variables defined as partial derivatives of the entropy density. More precisely, one defines $h:\mathcal Z \to \R^n$ as 
$h(\bu)  := \partial_u e(\bu) = (\partial_{u_0} e(\bu), \ldots, \partial_{u_n} e(\bu))$ for all $\bu\in \mathcal Z$. 
It turns out that, under appropriate assumptions, $h$ is a one-to-one mapping and thus for arbitrary $\bw$ its inverse satisfies $h^{-1}(\mathbf w) \in \mathcal Z$. 
This idea was first applied in \cite{BDPS2010} and later extended to more general systems in \cite{Juengel2015boundedness} and coined \emph{boundedness by entropy}. 

In our case, the method is not directly applicable due to the gradient term in the entropy density and one of our contributions is its extension through the analysis of an auxiliary variational problem.

Finally, let us remark that the question of regularity and uniqueness for cross-diffusion systems with entropic structure is mostly 
open except for a few works that, however, require additional assumptions, \cite{Zamponi2017_volumefilling, Zamponi2017_corrigendum,Berendsen2019}.


\subsubsection*{Contribution and structure of the paper}
In this article we prove the existence of global weak solutions to system \eqref{system} with energy \eqref{energy} and supplemented with appropriate initial- and boundary conditions.\medskip\\
 The novelty of our work lies in the following contributions.
	\begin{itemize}
		\item[(a)] This is, to the best of our knowledge, the first attempt to a cross-diffusion Cahn-Hilliard system. 
		\item[(b)] We are able to treat of an energy that \itshape only involves Cahn-Hilliard terms acting on $u_0$ \normalfont but not separately on the other species which yields to a transport term with low regularity in these equations. 
		This is done by an appropriate definition of weak solution and a careful analysis when performing the limit of an approximate time-discrete system.
		A similar situation has so far only been studies in the case of a non-degenerate mobility~\cite{Dai2017_degenerateginzburg}.
		\item[(c)] We generalize the boundedness-by-entropy to some case when one cannot explicitly invert $h$ but instead has to solve a system of elliptic PDEs with logarithmic non-linearities. 
		The literature on such systems is rather sparse (see \cite{Montenegro2009_lognonline,Alves2017_systemlognonlin}) but using variational methods we obtain existence of positive solutions.
	\end{itemize}

\medskip

This manuscript is organized as follows. In Section~\ref{sect2}, we give a precise definition of our mobility matrix, introduce our notion of weak solution and state the main existence theorem. 
The proof is based on 
the introduction of a regularized time discrete approximate problem, depending on a positive time step $\tau$, which is presented in Section~\ref{sect3}. 
We derive a priori estimates and prove the existence of time-discrete iterates via a Schauder fixed point argument. 
Finally, in Section~\ref{sect4} we exploit the regularity properties obtained in Section~\ref{apriori-sect} in order to pass to the limit $\tau \to 0^+$ and obtain a solution to \eqref{system}.

We intend to study the sharp-interface limit of this model in a future work.

\section{Preliminaries and main result}
\label{sect2}
Let us first introduce some notation used in the manuscript, give a precise definition of \eqref{system} and state our notion of weak solution together with the main existence result.

\subsection{Notation}

We assume in all the sequel that $\Omega$ is an open, bounded subset of $\mathbb{R}^d$ with $d \leq 3$ so that the embedding $H^2(\Omega) \hookrightarrow L^\infty(\Omega)$ is compact and fix a final time $T>0$. 
By $\N^*:= \N \setminus \{0\}$ we denote the set of positive integers. For a vector $\mathbf a \in \R^n$, $\operatorname{diag}(\mathbf a )$ denotes the $n\times n$ matrix that has the components of $\mathbf a$ on its diagonal.

\medskip

For any $\psi, \phi \in H^2(\Omega)$, we denote by 
$$
\langle \phi, \psi\rangle_{H^2(\Omega)}:= \int_\Omega \phi \psi + \nabla \phi \cdot \nabla \psi +  \Delta \phi \Delta \psi  dx, 
$$
and by $\|\phi\|_{H^2(\Omega)}:= \sqrt{{\langle \phi, \phi \rangle_{H^2(\Omega)}}}$. Similarly, for all $l =0,1,2,$  and for all $\bphi = (\phi_i)_{1\leq i \leq n}, \bpsi = (\psi_i)_{1\leq i \leq n} \in (H^l(\Omega))^n,$  we denote by
\begin{align*}
\langle \bphi, \bpsi \rangle_{H^l(\Omega)^n} =  \sum_{i=1}^n \langle \phi_i, \psi_i \rangle_{H^l(\Omega)}, 
\end{align*}
with 
$$
\|\bphi\|_{(H^l(\Omega))^n}:= \sqrt{\langle \bphi, \bphi \rangle_{(H^l(\Omega))^n}}
$$
and, for all $\bphi = (\phi_i)_{1\leq i \leq n} \in (L^\infty(\Omega))^n$, we set
$$
\|\bphi\|_{(L^\infty(\Omega))^n}:= \sqrt{\sum_{i=1}^n \|\phi_i\|_{L^\infty(\Omega)}^2}.
$$
Finally, we define the mapping $\kappa : \R_+ \times \R_+ \to \R$ by
\begin{align}\label{def:kappa}
\kappa(a,b) = \left\{ 
\begin{array}{ll}
\frac{a}{1-b} & \quad \mbox{ if } 1 -b \neq 0, \\
0 & \quad \mbox{ otherwise}.\\
\end{array}\right.
\end{align}

\subsection{Cahn-Hilliard cross-diffusion system}
Let us present the system we consider in this article in full detail. Let $\eps>0$ and $\beta >0$ and for all $\bu = (u_0, \dots, u_n) \in \left(L^\infty(\Omega)\cap H^1(\Omega)\right) \times (L^\infty(\Omega))^n$ consider the energy
$$
E(\bu)= \int_\Omega \sum_{i=0}^n (u_i \ln u_i -u_i +1) + \frac{\eps}{2} |\nabla u_0|^2 + \beta u_0 (1-u_0) dx. 
$$
For all $0\leq i \leq n$ let us introduce the chemical potentials, defined (on the formal level at this point) via 
\[
\mu_i = D_{u_i} E(\bu) = \ln u_i \qquad \forall \, i= 1, \dots, n,
\]
as well as
\[
\mu_0 = D_{u_0}E(\bu) = \ln u_0 - \eps \Delta u_0 + \beta (1- 2u_0),
\]
so that $\mu:= (\mu_0, \mu_1, \dots, \mu_n) = DE(\bu)$. Let us also introduce the auxiliary variables
\begin{equation}
\label{wi}
w_i := \ln u_i - \ln u_0 \qquad \forall \, i= 1, \dots, n,
\end{equation}
and 
\begin{equation*}
w_0 := -\eps \Delta u_0 + \beta(1-2u_0).
\end{equation*}
To specify the mobility matrix let, for $0\leq i \neq j \leq n$,  $K_{ij}$ denote some positive real number satisfying $K_{ij}= K_{ji}$. Then for $u \in \mathbb{R}^{n+1}$, let $M(\bu):= \left( M_{ij}(\bu)\right)_{0\leq i,j \leq n} \in \mathbb{R}^{(n+1)\times (n+1)}$ be the matrix 
\begin{equation}
\label{mobility}
\begin{aligned}
M_{ij}(\bu)&:= - K_{ij}u_i u_j \qquad &&\forall \, i\neq j= 0, \dots, n, \\
M_{ii}(\bu)&:= \sum_{0\leq j \neq i \leq n} K_{ij}u_iu_j && \forall \, i= 0, \dots, n.
\end{aligned}
\end{equation}
With these definitions, system \eqref{system} can be written, formally, in the scalar form: 
\begin{equation}
\label{eqi}
\begin{split}
 \partial_t u_i  & = \divergenz \biggl( \sum_{1\leq j \neq i \leq n} K_{ij}u_i u_j \nabla (\mu_i - \mu_j) + K_{i0}u_i u_0 \nabla (\mu_i - \mu_0)\biggr)  \\
 & = \divergenz \biggl( \sum_{1\leq j \neq i \leq n} K_{ij}u_i u_j \nabla (w_i - w_j) + K_{i0}u_i u_0 \nabla (w_i -w_0) \biggr)  \\
  & = \divergenz \biggl( \sum_{1\leq j \neq i \leq n} K_{ij} (u_j\nabla u_i - u_i\nabla u_j) + K_{i0}(u_0 \nabla u_i - u_i \nabla u_0) - K_{i0} u_i u_0\nabla w_0 \biggr),  \\
\end{split}
\end{equation}
for $1\leq i \leq n$ and
\begin{equation}
\label{eq0}
\begin{split}
 \partial_t u_0   & =  \divergenz \left( \sum_{1\leq i \leq n} K_{i0}u_i u_0 \nabla (\mu_0 - \mu_i)\right).  \\
  & =\div\left( \sum_{1\leq i \leq n} K_{i0}u_i u_0 \nabla (w_0 - w_i)\right)  \\
 & =\div\left( \sum_{1\leq i \leq n} K_{i0}(u_i \nabla u_0 - u_0 \nabla u_i) + K_{i0}u_i u_0\nabla w_0 \right).  \\
\end{split}
\end{equation}  
From this set of equations it is clear, at least formally, that 
\begin{align}\label{eq:sum_one}
\partial_t\left( \sum_{i=0}^n u_i \right) = 0.
\end{align}
Let us introduce an initial condition
\begin{align*}
\bu^0=(u_0^0, \ldots, u_n^0) \in H^2(\Omega; \R^{n+1})
\end{align*} 
of the system which is assumed to satisfy 
\begin{equation}
\label{in-cond}
u_i^0(x) \ge 0 \quad \forall \, 0 \le i \le n, \quad \sum_{i=0}^n u_i^0(x)= 1, \quad \text{and} \quad {\bu}(0, x)= {\bu}^0(x)  
\end{equation}
for a.e. $x \in \Omega$. In view of \eqref{eq:sum_one} we expect that solutions to system \eqref{system} satisfy 
\begin{equation}
\label{constr}
u_0 = 1 -\sum\limits_{i=1}^n u_i, \quad \text{ a.e. in } (0,T) \times \Omega,
\end{equation}
and it can be easily checked that, if $\bu$ satisfies (\ref{eqi}) and (\ref{constr}), then  necessarily  (\ref{eq0}) has to be satisfied as well. We make a last remark. As $0 \leq u_i \leq 1- u_0= \sum\limits_{j=1}^n u_j$ for all $1\leq i \leq n$, denoting 
\begin{align}\label{eq:def_kappa_i}
\kappa_i(t,x) := \kappa(u_i(t,x),u_0(t,x)),
\end{align}
it holds that 
	\[
	u_0 u_i \nabla w_0 = \frac{u_i}{1-u_0} u_0 (1-u_0) \nabla w_0  =  \kappa_i J,
	\]
 where $J := u_0(1-u_0) \nabla w_0$. 
 
\medskip

Then, supplementing this set of equations with no-flux boundary conditions, we obtain that $( (u_i)_{0\leq i \leq n}, J)$ is a solution to 
	\begin{equation}\label{eq:finalsys}
	\begin{aligned}
	&\partial_t u_i = \divergenz \biggl( \sum_{1\leq j \neq i \leq n} K_{ij} (u_j\nabla u_i - u_i\nabla u_j) + K_{i0}(u_0 \nabla u_i - u_i \nabla u_0) - K_{i0} \kappa_i J \biggr)  \, && \mbox{ in } (0,T)\times \Omega, \\
	&  u_0= 1 - \sum_{i=1}^n u_i &&  \mbox{ in } (0,T)\times \Omega,\\
	  &J= u_0 (1-u_0) \nabla \left( -\eps \Delta u_0 + \beta(1-2u_0)\right) && \mbox{ in } (0,T)\times \Omega,\\
	& \biggl( \sum_{1\leq j \neq i \leq n}  K_{ij} (u_j\nabla u_i - u_i\nabla u_j) + K_{i0}(u_0 \nabla u_i - u_i \nabla u_0) - K_{i0} \kappa_i J \biggr)  \cdot {\bf n} = 0 &&  \mbox{ in } (0,T)\times \partial \Omega,\\
	& u_i(0,\cdot) = u_i^0  &&  \mbox{ in } \Omega,
	 \end{aligned}
	\end{equation}
where ${\bf n}$ denotes the normal unit vector pointing outwards the domain $\Omega$.

\subsection{Notion of weak solution and main result}
The aim of our work is to prove the existence of a weak solution to system (\ref{eq:finalsys}) in the following sense.
\begin{e-definition}
\label{def1}
We say that $((u_i)_{0\leq i \leq n}, J) $ is a weak solution to \eqref{eq:finalsys} if 
\begin{enumerate}
\item  $0 \le u_i \le 1$ for every $ i= 0, \dots, n$;

\item $\sum\limits_{i=0}^n u_i = 1$ a.e. in $(0,T)\times \Omega$;  

\item $u_i \in L^2((0,T); H^1(\Omega))$ for all $1\leq i \leq n$;

\item $u_0 \in L^2((0,T); H^2(\Omega))$;
\item  $ \partial_t u_i \in L^2((0, T); (H^{1}(\Omega))')$ for all $0\leq i \leq n$;

\item $u_i(0,\cdot) = u_i^0$ for all $0\leq i \leq n$; 

\item $J \in (L^2((0,T)\times \Omega))^d$; 

\item $J = (1-u_0) u_0 \nabla \left(-\eps \Delta u_0 + \beta (1-2u_0)\right)$ in the following weak sense
$$
\int_0^T \int_\Omega J \cdot \bta  = - \int_0^T \int_\Omega \left(-\eps \Delta u_0 + \beta(1-2u_0)\right) \divergenz ((1-u_0) u_0 \bta) dx dt
$$
for all $\bta \in L^2((0,T); (H^1(\Omega))^d) \cap L^\infty((0,T)\times \Omega; \mathbb{R}^d)$ which satisfy $\bta \cdot {\bf n} = 0$ on $\partial \Omega \times (0,T)$;

\item  for all $1\leq i \leq n$, for all $\phi_i \in L^2((0,T); H^1(\Omega))$,
\begin{equation*}
\begin{split}
&\int_0^T\langle \partial_t u_i, \phi_i\rangle_{(H^1(\Omega))', H^1(\Omega)} dt \\
&=  -\int_0^T\int_{\Omega}  \l[\sum_{1 \le j \ne i \le n} K_{ij} (u_j\nabla u_i - u_i \nabla u_j) + K_{i0}\left(u_0\nabla u_i - u_i \nabla u_0 -   \kappa_i J \right)\r] \cdot \nabla \phi_i dx dt, 
\end{split}
\end{equation*}
where $\kappa_i(t,x) := \kappa(u_i(t,x),u_0(t,x)),$ with $\kappa$ defined in \eqref{def:kappa}. 

\end{enumerate}
\end{e-definition}
Note that due to $\kappa$, our definition of weak solution is related (but stronger) than the one introduced in \cite{Dai2016} for a scalar, degenerate Cahn-Hilliard equation.
Our main result is then the following.
\begin{theorem}\label{thm:existence}
Let $\bu^0=(u_0^0, \dots, u_n^0) \in H^2(\Omega; \R^{n+1})$ be an initial condition satisfying \eqref{in-cond}. Then, there exists at least one weak solution $\bu$ to \eqref{eq:finalsys} in the sense of Definition~\ref{def1}.
\end{theorem}

The rest of the article is devoted to the proof of Theorem~\ref{thm:existence}  which is structured as follows. 

We first prove the existence of solutions to a regularized time discrete version of system~\eqref{eq:finalsys}. The proof of the existence of solutions to this auxiliary problem is the object of Section~\ref{sect3} and is done using Schauder's fixed point theorem and an extension of the boundedness-by-entropy method, while Section~\ref{apriori-sect} is dedicated to estimates on various norms of such solutions. Finally, these estimates enable us to identify the limit of
the solution to the auxiliary problem as the time step goes to $0^+$ as a weak solution to (\ref{eq:finalsys}) 
in the sense of Definition~\ref{def1}. This last step is detailed in Section~\ref{sect4}.

\section{Existence of solutions to a regularized discrete in time system}
\label{sect3}
For further use we introduce the sets 
$$
\mathcal A:= \left\{\bu:= (u_i)_{1\leq i \leq n} \in (L^\infty(\Om))^n: \, u_i \ge 0, \, i=1,\ldots, n, \;   u_0:= 1 - \sum_{i=1}^n u_i  \geq 0\right\},
$$
and  
$$
\mathcal{B} := \left\{ \bphi = (\phi_i)_{1\leq i \leq n} \in (L^\infty(\Omega))^n \; : \; \phi_0:= - \sum_{i=1}^n \phi_i \in H^1(\Omega) \right\}.
$$
Let us point out here that  $\mathcal A$ is a closed convex non-empty subset of $(L^\infty(\Omega))^n$. Moreover, it is clear from the definition that for $\bu \in \mathcal A$ every $u_i$ satisfies the box constraints $ 0 \le u_i \le 1$, $i= 0, \dots, n$. 

The aim of this section is to prove the existence of a solution to a time-discrete regularized version of the system introduced in the previous section. 
More precisely, for every positive time step $\tau>0$, we want to give a rigorous sense to a regularized semi-discretization of our system formally defined as follows. 
For all $p\in \mathbb{N}$, given $\bu^p:=(u_1^p, \dots, u_n^p) \in \mathcal A \cap (H^2(\Omega))^n$  we look for a set of functions $\bu^{p+1}:=(u_1^{p+1}, \dots, u_n^{p+1})\in \mathcal A$ that is weak solution to the following nonlinear system: 
	\begin{equation}
	\label{dis-sys1-00}
	\begin{split}
	\into \frac{u_i^{p+1}- u_i^p}{\tau} \phi_i dx &= -\into \biggl(\sum_{1 \le j \ne i \le n} K_{ij} u_i^{p+1} u_j^{p+1} \nabla (w_i^{p+1}- w_j^{p+1})\\
	& \quad \quad+ K_{i0} u_i^{p+1} u_0^{p+1} \nabla (w_i^{p+1}- w_0^{p+1/2}) \biggr) \cdot \nabla \phi_i dx \\
	& \quad -\tau \langle w_i^{p+1}- w_0^{p+1/2}, \phi_i \rangle_{H^2(\Omega)},
	\end{split}
	\end{equation}
for all $1\leq i \leq n$, where $$
u_0^{p+1}:= 1 - \sum_{i=1}^n u_i^{p+1}, \quad u_0^p = 1 - \sum_{i=1}^n u_i^p,
$$
	\begin{equation}
	\label{w0p}
	w_0^{p+1/2}:= -\eps \Delta u_0^{p+1}+ \beta (1- 2u_0^p),
	\end{equation}
and
$$
w_i^{p+1}:= \ln u_i^{p+1} - \ln u_0^{p+1}, \quad i = 1, \dots, n.
$$
Let us emphasize that we use a semi-implicit discretization as we consider the terms arising from the concave part of the entropy at the previous time step $p$. We will see below that this ensures that the discrete energy is non-increasing.

To give a rigorous sense to this nonlinear system  we will make use of a fixed-point argument. First of all, let us point out that, defining 
	\begin{equation*}
	\bar w_i^{p+1}:= w_i^{p+1}- w_0^{p+1/2} =\ln u_i^{p+1} - \ln u_0^{p+1} + \eps \Delta u_0^{p+1} - \beta (1- 2u_0^p),
	\end{equation*}
for all $1\leq i \leq n$,  system \eqref{dis-sys1-00} boils down to
	\begin{equation}
	\label{dis-sys2}
	\begin{split}
	\into \frac{u_i^{p+1}- u_i^p}{\tau} \phi_i dx&= -\into \biggl(\sum_{1 \le j \ne i \le n} K_{ij} u_i^{p+1} u_j^{p+1} \nabla (\bar w_i^{p+1}- \bar w_j^{p+1}) + K_{i0} u_i^{p+1} u_0^{p+1} \nabla \bar w_i^{p+1} \biggr) \cdot \nabla \phi_i dx \\
	& \quad -\tau  \langle \bar w_i^{p+1}, \phi_i \rangle_{H^2(\Omega)}.\\
	\end{split}
	\end{equation}
The auxiliary variables $\bar \bw^{p+1}=(\bar w_1^{p+1}, \dots, \bar w_n^{p+1})$ will play a central role in the proof of the existence of solutions to this semi-discretized regularized system. We have the following result.
\begin{theorem}\label{thm:existence_time_discrete}
 Let $\tau >0$ be a discrete time step, let $p\in \mathbb{N}$, and let $\bu^p \in \mathcal A \cap (H^2(\Omega))^n$. Then, there exists a solution   $(\bu^{p+1}, \bar \bw^{p+1}) \in (\mathcal A \cap (H^2(\Omega))^n) \times (H^2(\Omega))^n$ to the following coupled system: for all $1\leq i \leq n$, for all $\phi_i \in  H^2(\Omega)$,
 	\begin{equation}
	\label{dis-sys1}
	\begin{split}
	\into \frac{u_i^{p+1}- u_i^p}{\tau} \phi_i dx &= -\into \biggl(\sum_{1 \le j \ne i \le n} K_{ij} u_i^{p+1} u_j^{p+1} \nabla (\bar{w}_i^{p+1}- \bar{w}_j^{p+1})+ K_{i0} u_i^{p+1} u_0^{p+1} \nabla \bar{w}_i^{p+1} \biggr) \cdot \nabla \phi_i dx \\
	& \quad -\tau \langle \bar{w}_i^{p+1}, \phi_i \rangle_{H^2(\Omega)},
	\end{split}
	\end{equation}
 and for all $\bpsi = (\psi_i)_{1\leq i \leq n} \in \mathcal B \cap (L^\infty(\Omega))^n  $,  
\begin{equation}\label{eq:w_to_u_weak0}
\sum_{i=1}^n \into (\ln u^{p+1}_i- \ln u^{p+1}_0)\psi_i + \eps \nabla u^{p+1}_0\cdot \nabla \psi_0 dx = \sum_{i=1}^n \into \left( \bar w^{p+1}_i + \beta(1- 2 u_0^p)\right)\psi_i dx,
\end{equation}
where $u_0^p$ is given by \eqref{constr}.
 Moreover, the function $\bu^{p+1}$ satisfies the following property: there exists $\delta_p>0$ such that 
\begin{equation}\label{eq:positiv}
u_i^{p+1}\geq \delta_p, \quad \mbox{ for all }1\leq i \leq n, \;\mbox{ and }  \;u_0^{p+1}:= 1 - \sum_{i=1}^n u_i^{p+1} \geq \delta_p,\quad\text{ a.e. in } (0,T)\times\Omega.
\end{equation}
\end{theorem}
\begin{remark}\label{rem:CL}

The weak formulation \eqref{eq:w_to_u_weak0} implies that, for all $1\leq i \leq n$, 
\begin{equation}\label{eq:eqnew}
\into (\ln u^{p+1}_i- \ln u^{p+1}_0)\psi_i - \eps \nabla u^{p+1}_0\cdot \nabla \psi_i dx =  \into \left( \bar w^{p+1}_i + \beta(1- 2 u_0^p)\right)\psi_i dx
\end{equation}
for all $\psi_i \in L^\infty(\Omega) \cap H^1(\Omega)$.
Besides, since $\ln u^{p+1}_i$, $\ln u^{p+1}_0$, $\bar w^{p+1}_i$, and $\beta(1- 2 u_0^p)$ belong to $L^\infty(\Omega)$, the first three thanks to (\ref{eq:positiv}) and the last one by assumption, and since the set $L^\infty(\Omega) \cap H^1(\Omega)$ is dense in $H^1(\Omega)$,
we obtain that (\ref{eq:eqnew}) holds for all $\psi_i \in H^1(\Omega)$. As a consequence, $u_0^{p+1}$ is the unique solution in $H^1(\Omega)$ to the problem 
	\begin{equation*} 
	\begin{aligned}
	- &\Delta u^{p+1}_0=  \bar w^{p+1}_i + \beta(1- 2 u_0^p) - \ln u^{p+1}_i + \ln u^{p+1}_0 \quad && \mbox{ in }\mathcal D'(\Omega),\\
	&\nabla u_0^{p+1} \cdot {\bf n} = 0 && \mbox{ on }\partial \Omega. 
	\end{aligned}
	\end{equation*}
\end{remark}

\medskip

From now on and in all the rest of Section~\ref{sect3}, we fix $\tau>0$, $\bu^p:=(u_1^p, \dots, u_n^p)\in \mathcal A \cap (H^2(\Omega))^n$ and denote by $u_0^p:= 1 - \sum\limits_{i=1}^n u_i^p$. 

\medskip

The proof of Theorem~\ref{thm:existence_time_discrete} makes use of Schauder's fixed point theorem as follows. We first show that for any $\tilde \bu = (\tilde u_1, \dots, \tilde u_n)\in \mathcal A$ there exists a unique solution 
$\bar \bw  = (\bar w_1, \dots, \bar w_n)\in (H^2(\Omega))^n$ to the linearised problem: for all $1\leq i \leq n$ and all $\phi_i\in H^2(\Omega)$, 
 	\begin{equation}
	\label{dis-sys12}
	\begin{split}
	\into \frac{\tilde u_i- u_i^p}{\tau} \phi_i dx &= -\into \biggl(\sum_{1 \le j \ne i \le n} K_{ij} \tilde u_i \tilde u_j \nabla (\bar w_i- \bar w_j)+ K_{i0} \tilde u_i \tilde u_0 \nabla \bar w_i \biggr) \cdot \nabla \phi_i dx \\
	& \quad -\tau \langle \bar w_i^{p+1}, \phi_i \rangle_{H^2(\Omega)},
	\end{split}
	\end{equation}
	with $\tilde u_0:= 1 - \sum\limits_{i=1}^n \tilde u_i$. We then prove that the map $\mathcal S_1: \mathcal A \to (H^2(\Omega))^n$ which associates  to $\tilde \bu \in \mathcal A$ the unique solution $\bar \bw$ to 
(\ref{dis-sys12}) is continuous. This is the object of Section~\ref{subsec:ex_w}. 

\medskip

We then show that for all $\bar \bw \in (H^2(\Omega))^n$, there exists a unique solution $\bu  \in \mathcal A \cap (H^2(\Omega))^n$ to
\begin{equation}\label{eq:w_to_u_weak01}
\sum_{i=1}^n \into (\ln u_i- \ln u_0)\psi_i + \eps \nabla u_0\cdot \nabla \psi_0 dx = \sum_{i=1}^n \into \left( \bar w_i + \beta(1- 2 u_0^p)\right)\psi_i dx,
\end{equation}
for all $\bpsi = (\psi_i)_{1\leq i \leq n} \in \mathcal B \cap (L^\infty(\Omega))^n $, with $u_0$ given by \eqref{constr}. Problem (\ref{eq:w_to_u_weak01}) 
is to be interpreted as a weak formulation associated to the relation
$$
\ln u_i - \ln u_0 = \bar w_i -\eps \Delta u_0 + \beta (1-2u_0^p).
$$
The map $\mathcal S_2:  (H^2(\Omega))^n \to \mathcal A$ which to each $\bar \bw \in (H^2(\Omega))^n$ 
associates  the unique solution $\bu\in \mathcal A$ to (\ref{eq:w_to_u_weak01}) is then shown to be continuous. These results are proved in Section~\ref{subsec:u_from_w}.

\medskip

We finally conclude by showing that the map $\mathcal S = \mathcal S_2 \circ \mathcal S_1: \mathcal A \to \mathcal A$ is such that $\mathcal S(\mathcal A)$ is 
a relatively compact subset of $(L^\infty(\Omega))^n$, so that Schauder's fixed point theorem can be used. This is the object of Section~\ref{sec:proof1}.

\subsection{Definition and continuity of the map $\mathcal S_1$}\label{subsec:ex_w}

\begin{lemma}
	\label{thm3}
	For any $\tilde \bu  \in \mathcal A$, there exists a unique solution $\bar \bw  \in (H^2(\Om))^n$ to the problem: for all $1\leq i \leq n$, 
	for all $\phi_i \in H^2(\Omega)$, 
	\begin{equation}
	\label{discr1-1}
	\begin{split}
	\into \frac{\tilde u_i-  u_i^p}{\tau} \phi_i dx &= -\into \biggl(\sum_{1 \le j \ne i \le n} K_{ij} \tilde u_i \tilde u_j \nabla (\bar{w}_i- \bar{w}_j)+ K_{i0} \tilde u_i \tilde u_0 \nabla \bar{w}_i \biggr) 
	\cdot \nabla \phi_i dx \\
	& \qquad -\tau \langle \bar{w}_i, \phi_i \rangle_{H^2(\Omega)},
	\end{split}
	\end{equation}
	where $\tilde u_0 $ satisfies \eqref{constr}.  Furthermore, there exists a constant $M_0>0$, depending only on $n$, $\tau$, and $\Omega$, such that 
	\begin{equation}\label{eq:bound}
	\|\bar \bw\|_{(H^2(\Omega))^n} \leq M_0.
	\end{equation}
\end{lemma}

\begin{proof}
We fix $\tilde \bu:=(\tilde u_i)_{1\leq i \leq n} \in \mathcal A$ and  introduce the matrices 
	$G(\tilde \bu) := (G_{ij}(\tilde \bu))_{1\leq i,j \leq n} $ and $ H(\tilde \bu) := (H_{ij}(\tilde \bu))_{1\leq i,j \leq n} \in \R^{n\times n}$ defined by
	\begin{equation*}
	\begin{aligned}
	G_{ij}({\tilde \bu})&:= - K_{ij}\tilde u_i \tilde u_j \qquad &&\forall \, i\neq j= 1, \dots, n, \\
	G_{ii}({\tilde \bu})&:= \sum_{1\leq j \neq i \leq n} K_{ij}\tilde u_i\tilde u_j && \forall \, i= 1, \dots, n,\\
	\end{aligned}
	\end{equation*}
	and
	\[
	H(\tilde \bu)=  \operatorname{diag}(K_{10} \tilde u_1 \tilde u_0, \ldots, K_{n0} \tilde u_n \tilde u_0).
	\]
	Then, system \eqref{discr1-1} can be equivalently written as  
	\begin{equation}
	\label{discr1.1}
	\begin{split}
	-\frac{1}{\tau} \into 
	(\tilde \bu- \bu^p )\cdot \bphi dx &= \into \nabla \bphi \cdot G(\tilde \bu) \nabla \bar \bw dx + \into \nabla \bphi \cdot  H(\tilde \bu) \nabla \bar \bw dx + \tau  \langle \bphi, \bar \bw \rangle_{(H^2(\Omega))^n},
	\end{split}
	\end{equation}
 for all $\bphi \in (H^2(\Omega))^n$. 	Let us point out that   
	\begin{equation}\label{eq:bound4}
	0 \leq G(\tilde \bu) \leq n\overline{K} {\rm I}_n \quad \mbox{ and } 0\leq H(\tilde \bu) \leq \overline{K} {\rm I}_n
	\end{equation}
almost everywhere in $\Omega$, in the sense of symmetric matrices, with $\overline{K}:= \max\limits_{0\leq i \neq j \leq n} K_{ij}$ and ${\rm I}_n$ being the identity matrix of $\mathbb{R}^{n\times n}$.
	The existence and uniqueness of a solution to \eqref{discr1.1} is then a consequence of Lax-Milgram's theorem. In particular, taking $\bphi = \bar \bw$ in \eqref{discr1.1} gives
	\begin{align*}
	\tau \|\bar \bw\|_{(H^2(\Omega))^n}^2 & \leq \frac{1}{\tau}\sum_{i=1}^n \|\tilde u_i - u_i^p\|_{L^2(\Omega)}\|\bar w_i\|_{L^2(\Omega)}\\
	& \leq \frac{1}{\tau}\left( \sum_{i=1}^n \|\tilde u_i - u_i^p \|_{L^2(\Omega)}^2\right)^{1/2} \|\bar \bw\|_{H^2(\Omega)}. \\
	\end{align*}
	Since $\tilde \bu$ and $\bu^p$ belongs to $\mathcal A$, this implies that 
	$$
	\|\bar \bw\|_{(H^2(\Omega))^n} \leq \frac{1}{\tau^2}2 \sqrt{n|\Omega|},
	$$
	which yields estimate (\ref{eq:bound}). 
\end{proof}

Let us denote by $\mathcal S_1: \mathcal A \subset (L^\infty(\Omega))^n \to (H^2(\Omega))^n$ 
the application that associates to each $\tilde \bu \in \mathcal A$   the unique solution $\bar \bw$ to \eqref{discr1-1}. We have the following result.

\begin{lemma}\label{lem:S1cont}
The map $\mathcal S_1: \mathcal A \subset (L^\infty(\Omega))^n \to (H^2(\Omega))^n$ is continuous. 
\end{lemma}
\begin{proof}
 Let $\tilde \bu^1, \tilde \bu^2 \in \mathcal A$ and set $\bar \bw^1 = \mathcal S_1(\tilde \bu^1)$ as well as $\bar \bw^2 = \mathcal S_1(\tilde \bu^2)$. 
For all $1\leq i\neq j \leq n$ we have  
	\begin{align*}
	G_{ij}(\tilde \bu^1) - G_{ij}(\tilde \bu^2) & = - K_{ij} \left[ \tilde u_i^1 (\tilde u_j^1 - \tilde u_j^2) + \tilde u_j^2( \tilde u_i^1- \tilde u_i^2)\right],\\
	G_{ii}(\tilde \bu^1) - G_{ii}(\tilde \bu^2) & = \sum_{1\leq j \neq i \leq n} K_{ij}\left[ \tilde u_i^1 (\tilde u_j^1 - \tilde u_j^2) + \tilde u_j^2( \tilde u_i^1- \tilde u_i^2)\right],\\
	H_{ii}(\tilde \bu^1) - H_{ii}(\tilde \bu^2) & = K_{i0} \left[ \tilde u_i^1 (\tilde u_0^1 - \tilde u_0^2) + \tilde u_0^2( \tilde u_i^1- \tilde u_i^2)\right],\\
	\end{align*}
	which yield the Lipschitz estimates
	\begin{align*}
	\|G_{ij}(\tilde \bu^1) - G_{ij}(\tilde \bu^2)\|_{L^\infty(\Omega)} & \leq \overline{K}  \left( \|\tilde u_j^1 - \tilde u_j^2\|_{L^\infty(\Omega)} + \|\tilde u_i^1- \tilde u_i^2\|_{L^\infty(\Omega)}\right),\\
	\|G_{ii}(\tilde \bu^1) - G_{ii}(\tilde \bu^2)\|_{L^\infty(\Omega)} & = \overline{K}\left[(n-1)\|\tilde u_i^1- \tilde u_i^2\|_{L^\infty(\Omega)} + \sum_{1\leq j \neq i \leq n} \|\tilde u_j^1 - \tilde u_j^2\|_{L^\infty(\Omega)} \right],\\
	\|H_{ii}(\tilde \bu^1) - H_{ii}(\tilde \bu^2)\|_{L^\infty(\Omega)} & = \overline{K} \left[ \|\tilde u_0^1 - \tilde u_0^2\|_{L^\infty(\Omega)} + \|\tilde u_i^1- \tilde u_i^2\|_{L^\infty(\Omega)}\right].\\
	\end{align*}
	Since $\|\tilde u_0^1 - \tilde u_0^2\|_{L^\infty(\Omega)} \leq \sum\limits_{i=1}^n  \|\tilde u_i^1- \tilde u_i^2\|_{L^\infty(\Omega)}$,  there exists a constant $C>0$,  only depending on $n$ and $\overline{K}$, such that
	\begin{equation}\label{eq:bound1}
	- C \left( \sum_{i=1}^n \|\tilde u_i^1- \tilde u_i^2\|_{L^\infty(\Omega)}\right) {\rm I} \leq G(\tilde \bu^1) - G(\tilde \bu^2) \leq C \left( \sum_{i=1}^n \|\tilde u_i^1- \tilde u_i^2\|_{L^\infty(\Omega)}\right) {\rm I}
	\end{equation}
	and 
	\begin{equation}\label{eq:bound2}
	- C \left( \sum_{i=1}^n \|\tilde u_i^1- \tilde u_i^2\|_{L^\infty(\Omega)}\right) {\rm I} \leq H(\tilde \bu^1) - H(\tilde \bu^2) \leq C \left( \sum_{i=1}^n \|\tilde u_i^1- \tilde u_i^2\|_{L^\infty(\Omega)}\right) {\rm I},
	\end{equation}
		almost everywhere in $\Omega$, in the sense of symmetric matrices.
	Then, for all $\bphi \in (H^2(\Omega))^n$,
\begin{align*}
&\frac{1}{\tau} \into 
(\tilde \bu^1- \tilde \bu^2 )\cdot \bphi dx \\
&= -\into \nabla \bphi \cdot \left( G(\tilde \bu^1) \nabla \bar \bw^1 - G(\tilde \bu^2) \nabla \bar \bw^2\right) dx -\into \nabla \bphi \cdot  \left(H(\tilde \bu^1) \nabla \bar \bw^1 - H(\tilde \bu^2) \nabla \bar \bw^2 \right) dx \\
& \qquad - \tau  \langle \bphi, \bar \bw^1 - \bar \bw^2 \rangle_{(H^2(\Omega))^n}.
\end{align*}
Choosing $\bphi = \bar \bw^1 - \bar \bw^2$ in the above equality and using \eqref{eq:bound}, \eqref{eq:bound1}, \eqref{eq:bound2}, and \eqref{eq:bound4} gives the existence of a constant $C'>0$, depending only on $n$, $\overline{K}$, $|\Omega|$, and $\tau$ such that
\begin{align*}
 & \tau \|\bar \bw^1 - \bar \bw^2\|_{(H^2(\Omega))^n}^2 \\
 & =
 - \frac{1}{\tau} \into (\tilde \bu^1- \tilde \bu^2 )\cdot (\bar \bw^1 - \bar \bw^2) dx  \\
 & \qquad -\into \nabla  (\bar \bw^1 - \bar \bw^2) \cdot \left( G(\tilde \bu^1) \nabla (\bar \bw^1 - \bar \bw^2) \right) dx
- \into \nabla  (\bar \bw^1 - \bar \bw^2) \cdot \left( G(\tilde \bu^1)  - G(\tilde \bu^2) \right)\nabla \bar \bw^2 dx \\
 & \qquad -\into \nabla  (\bar \bw^1 - \bar \bw^2) \cdot \left( H(\tilde \bu^1) \nabla (\bar \bw^1 - \bar \bw^2) \right) dx
  - \into \nabla  (\bar \bw^1 - \bar \bw^2) \cdot \left( H(\tilde \bu^1)  - H(\tilde \bu^2) \right)\nabla \bar \bw^2 dx\\
& \leq - \frac{1}{\tau} \into (\tilde \bu^1- \tilde \bu^2 )\cdot (\bar \bw^1 - \bar \bw^2)dx - \into \nabla  (\bar \bw^1 - \bar \bw^2) \cdot \left( G(\tilde \bu^1)  - G(\tilde \bu^2) \right)\nabla \bar \bw^2 dx\\ 
& \qquad - \into \nabla  (\bar \bw^1 - \bar \bw^2) \cdot \left( H(\tilde \bu^1)  - H(\tilde \bu^2) \right)\nabla \bar \bw^2 dx \\
& \leq C' \left( \sum_{i=1}^n \|\tilde u_i^1- \tilde u_i^2\|_{L^\infty(\Omega)}\right) \|\bar \bw^1 - \bar \bw^2\|_{(H^2(\Omega))^n}.\\
\end{align*}
Thus,
$$
\|\bar \bw^1 - \bar \bw^2\|_{(H^2(\Omega))^n} \leq C' \left( \sum_{i=1}^n \|\tilde u_i^1- \tilde u_i^2\|_{L^\infty(\Omega)}\right),
$$
which yields the continuity of  $\mathcal S_1$. 
\end{proof}

\subsection{Definition and continuity of the map $\mathcal S_2$}\label{subsec:u_from_w}
%
%

The aim of this section to prove the existence and uniqueness of a solution $\bu  \in \mathcal A$ to the problem
\begin{equation}
\label{discr2}
\begin{split}
\bar w_i= \ln u_i- \ln u_0+ \eps \Delta u_0- \beta(1- 2 u_0^p), \quad i=1,\ldots, n,
\end{split}
\end{equation}
when  $\bar \bw \in (H^2(\Omega))^n$ is given. 

\medskip

In the case when $\eps = \beta = 0$, there is an algebraic relation which allows to explicitly express $\bu$ in terms of  $\bar \bw$ and ensures that $\bu \in \mathcal A$ (the boundedness-by-entropy method \cite{BDPS2010,Juengel2015boundedness}).
In our case, the situation is more involved, since, due to gradient term in the entropy, 
the densities $u_i$ are solutions to the nonlinear coupled elliptic system \eqref{discr2}.

More precisely, we will identify  the solution $\bu$ to \eqref{discr2} as the unique weak solution in $\mathcal A \cap \mathcal B$ to the variational problem
\begin{align}\label{eq:w_to_u_weak}
\sum_{i=1}^n \into (\ln u_i- \ln u_0)\phi_i + \eps \nabla u_0\cdot \nabla \phi_0 dx = \sum_{i=1}^n \into \left( \bar w_i + \beta(1- 2 u_0^p)\right)\phi_i dx,
\end{align}
for all $\bphi  \in \mathcal B \cap (L^\infty(\Omega))^n $ which 
%
will be equivalently characterized as the unique solution to the minimization problem
\begin{equation}\label{eq:var_problem}
\min_{\bv \in \mathcal A}  F_{\bar \bw}(\bv)
\end{equation}
where for all $\bv \in \mathcal A$ we define
$$
F_{\bar \bw}(\bv) = 
\begin{cases}
 +\infty & \mbox{ if } \bv\notin \mathcal B\\
 \into \sum\limits_{i= 0}^n v_i \ln v_i+ \frac{\eps}{2} |\nabla v_0|^2- \sum\limits_{i=1}^n v_i f_i dx & \mbox{ otherwise},
\end{cases} 
$$
with  $f_i:= \bar w_i + \beta(1- 2 u_0^p)$ for all $1\leq i \leq n$.

\medskip

The goal of this section is to rigorously prove all these claims. To this aim, we will proceed into three steps: first we show that minimizers to (\ref{eq:var_problem}) exist, then that these minimizers are solutions to (\ref{eq:w_to_u_weak}),
and finally that the solution to \eqref{eq:w_to_u_weak} is unique.


\begin{lemma} For all $\bar \bw \in (H^2(\Omega))^n$,  problem \eqref{eq:var_problem} admits at least one minimizer $\bu \in \mathcal A$. 
\end{lemma}

\begin{proof} Let $\bar \bw \in (H^2(\Omega))^n$. For all $1\leq i \leq n$, let $f_i:= \bar w_i + \beta(1- 2 u_0^p) \in H^2(\Omega) \subset L^\infty(\Omega)$.
Let us first show that $F_{\bar \bw}$ is bounded from below on $\mathcal A$. Fix $\bu = (u_i)_{1\leq i \leq n} \in \mathcal A$. Since  $x\ln x- x+ 1 \ge 0$  for all $x \in [0, 1]$ and since $\displaystyle -\into u_i \geq -|\Om|$ for all $i=0,\ldots, n$, we have
	\begin{equation}\label{eq:lower_bound_log}
	\into \sum_{i=0}^n u_i \ln u_i dx \ge \into \sum_{i=0}^n (u_i- 1) dx= -n|\Om|.
	\end{equation}
	Moreover, 
\begin{equation}\label{eq:est2}
-\into u_i f_i dx \ge - \|f_i\|_{L^\infty(\Omega)} |\Om|,
\end{equation}
for all $1\leq i \leq n$.
Collecting these estimates gives the existence of a constant $C>0$, which only depends on $n$ and $\Omega$, such that 
	\begin{equation*}
	F_{\bar \bw}(\bu) \ge - C \left( 1 + \sum_{i=1}^n \|f_i\|_ {L^\infty(\Omega)}\right).
	\end{equation*}
This shows that $F_{\bar \bw}$ is bounded from below on $\mathcal A$. Thus, $\inf_{\mathcal A} F_{\bar \bw} > -\infty$. Besides, $\inf_{\mathcal A} F_{\bar \bw} \leq F_{\bar \bw}(\mathbf 0) = 0$, as $\mathbf 0 \in \mathcal A$.
Thus, there exists a minimizing sequence $({\bf u}^{(m)})_{m \in \N} \subset \mathcal A$ such that $(F_{\bar \bw}({\bf u}^{(m)}))_{m\in\mathbb{N}}$ is bounded and
	\[
	\lim_{m \to \infty}  F_{\bar \bw}({\bf u}^{(m)})= \inf_{\mathcal A} F_{\bar \bw}.
	\]
Using estimates (\ref{eq:lower_bound_log}) and (\ref{eq:est2}) we obtain that $\left( \left\| \nabla u_0^{(m)} \right\|_{L^2(\Omega)}\right)_{m\in \mathbb{N}}$ is bounded as well.
Thus, up to the extraction of a subsequence that we still denote by ${\bf u}^{(m)}$ for the sake of simplicity, there exists $\bu  \in \mathcal A \cap \mathcal B$ such that, as $m \to +\infty$,  
\begin{align*}
u_i^{(m)} & \mathop{\to}  u_i &\mbox{ weakly-* in }L^\infty(\Omega), \;\mbox{ strongly in }L^p(\Omega)\mbox{ for all }1\leq p<+\infty \mbox{ and a.e. in }\Omega, \\
u_0^{(m)} & \mathop{\to}  u_0:= 1 - \sum_{i=1}^n u_i &\mbox{ strongly in }L^p(\Omega)\mbox{ for all }1\leq p<+\infty \mbox{ and a.e. in }\Omega, \\
\nabla u_0^{(m)} & \mathop{\rightharpoonup}  \nabla u_0 &\text{ weakly in }(L^2(\Omega))^d.\\
\end{align*}
Since the function $[0,1] \ni x \mapsto x\ln x$ is continuous and bounded in $[0,1]$, the Lebesgue dominated convergence theorem yields that for all $0\leq i \leq n$  
$$
\into u_i^{(m)} \ln u_i^{(m)} dx \mathop{\to}  \into u_i\ln u_i dx,\quad \text{ as }m\to +\infty.
$$
Furthermore, it holds that 
$$
\into |\nabla u_0|^2 dx \leq \mathop{\liminf}_{m\to +\infty} \into |\nabla u_0^{(m)}|^2 dx,
$$
and finally 
$$
\into f_i u_i^{(m)} dx \mathop{\to} \into f_i u_i dx,\quad \text{ as }m\to +\infty.
$$
This implies
$$
F_{\bar \bw}(\bu) \leq \mathop{\liminf}_{m\to +\infty} F_{\bar \bw}(\bu^{(m)}) = \mathop{\inf}_{\mathcal A} F_{\bar \bw}, 
$$
which entails that $\bu$ is necessarily a minimizer of $F_{\bar \bw}$ on $\mathcal A$. 	
\end{proof}

\begin{lemma}\label{lem:pos_minimizers} For all $\bar \bw \in (H^2(\Omega))^n$  there exists $\delta_{\bar \bw} >0$ such that for any minimizer $\bu$ to 
\eqref{eq:var_problem} it holds
$$
 u_i \geq \delta_{\bar \bw} \quad \forall \, 1\leq i \leq n, \quad 1 - \delta_{\bar \bw} \ge u_0:= 1 - \sum_{i=1}^n u_i \geq \delta_{\bar \bw}, \quad \mbox{a.e. in }\Omega.  
$$
Besides, for all $N>0$, there exists $\delta >0$ which only depends on $n$, $\Omega$, $\tau$, $\beta$, and $N$, such that for all $\bar \bw \in (H^2(\Omega))^n$ with $\|\bar \bw\|_{(H^2(\Omega))^n} \leq N$ and
for any minimizer $\bu$ to \eqref{eq:var_problem}  it holds that
$$
 u_i \geq \delta \quad \forall \, 1\leq i \le n, \quad 1 - \delta \ge u_0:= 1 - \sum_{i=1}^n u_i \geq \delta \quad \mbox{a.e. in }\Omega.  
$$
\end{lemma}

\begin{proof}
Let $\bar \bw  \in (H^2(\Omega))^n$ and for all $1\leq i \leq n$ let us denote by $f_i := \bar w_i + \beta(1- 2 u_0^p) \in H^2(\Omega) \subset L^\infty(\Omega)$. Let $\bu $ be a minimizer of $F_{\bar \bw}$ on $\mathcal A$. 

\medskip

\bfseries Step~1: \normalfont Let us first show that there exists $1 > \delta >0$, which only depends on $n$, $\Omega$, $\beta$, $\tau$, and $\sum\limits_{i=1}^n \|\bar w_i\|_{L^\infty(\Omega)}$,
such that $\delta \leq u_0$ almost everywhere in $\Omega$. The precise value of $\delta$ will be specified later in the proof. 

\medskip

We reason by contradiction and assume that the Lebesgue measure of the set
$$
\mathcal M_\delta := \{ x \in \Omega \; : \; u_0(x) < \delta \}
$$ is positive. 
Now, let us define 
\begin{align} \nonumber 
u_0^\delta &:= \max (u_0, \delta),\\
\begin{split}\label{eq:u1u2delta}
u_i^\delta &:= u_i - (u_0^\delta - u_0)\frac{u_i}{1-u_0},\quad i=1,\ldots, n,
\end{split}
\end{align}
and $\bu^\delta := (u_1^\delta, \ldots, u_n^\delta)$. In (\ref{eq:u1u2delta}), since $1-u_0 = \sum\limits_{j=1}^n u_j\geq u_i \geq 0$, the function $\frac{u_i}{1-u_0}$ is well-defined almost everywhere using the convention that 
$\frac{u_i}{1-u_0} = 0$ as soon as $u_i = 0$. By definition, it holds that $1 \geq u_0^\delta \geq 0$ and $u_0^\delta + \sum\limits_{i=1}^nu_i^\delta = 1$. 
Furthermore,  $u_i^\delta(x) =0$ for all $x\in \Omega$ such that $u_i(x) = 0$. For all $x\in \Omega$ such that $u_i(x)>0$,  it follows that $1-u_0(x) \geq u_i(x)>0$ and 
$$
	u_i^\delta(x) = u_i(x)\left(1 - \frac{u_0^\delta(x) - u_0(x)}{1-u_0(x)}\right) \ge 0, \quad \text{ since } \quad \frac{u_0^\delta(x) - u_0(x)}{1-u_0(x)} \le \frac{1 - u_0(x)}{1-u_0(x)} = 1.
	$$
As a consequence,  $\bu^\delta \in \mathcal A$ and $u_0^\delta = 1 - \sum\limits_{i=1}^n u_i^\delta$. 
We now prove that for $\delta$ sufficiently small, $F_ {\bar \bw}(u_1^\delta,u_2^\delta) < F_{\bar \bw}(u_1,u_2)$. 
Indeed, using the convexity of the function $[0,1] \ni x \mapsto x\ln(x)$, the fact that $|\nabla u_0^\delta| \le |\nabla u_0|$ a.e. in $\Omega$ and   
that $u_i^\delta = u_i$ on $\mathcal M_\delta^c = \{x \in \Omega\; : \; u_0(x) \ge \delta \}$ yields 
	\begin{align}\nonumber
	&F_{\bar \bw}(\bu^\delta) - F_{\bar \bw}(\bu) \\\nonumber
	&\leq \int_{\mathcal M_\delta} \sum_{i=1}^n [u_i^\delta \ln u_i^\delta - u_i \ln u_i] + [u_0^\delta \ln u_0^\delta -u_0\ln u_0] - \sum_{i=1}^nf_i(u_i^\delta-u_i) dx \\ \nonumber
	&\leq \int_{\mathcal M_\delta} \sum_{i=1}^n [\ln u_i^\delta + 1](u_i^\delta - u_i) + [\ln u_0^\delta + 1](u_0^\delta - u_0) - \sum_{i=1}^n f_i(u_i^\delta-u_i) dx \\ \nonumber 
	&=\int_{\mathcal M_\delta} \sum_{i=1}^n-[\ln u_i^\delta + 1](u_0^\delta - u_0)\frac{u_i}{1-u_0} + [\ln u_0^\delta + 1](u_0^\delta - u_0) dx \\ \label{eq:terms_linear}
	&\qquad - \int_{\mathcal M_\delta} \sum_{i=1}^n f_i(u_i^\delta-u_i) dx.
	\end{align}
	To estimate the first term we note that 
	\begin{align*}
	u_i^\delta  = u_i - (u_0^\delta - u_0)\frac{u_i}{1-u_0}  = \frac{u_i}{1-u_0}(1-u_0^\delta),
	\end{align*}
	for all $1\leq i \leq n$. Therefore,
	\begin{align*}
	&	\int_{\mathcal M_\delta}  \sum_{i=1}^n-[\ln u_i^\delta + 1](u_0^\delta - u_0)\frac{u_i}{1-u_0} dx \\
	&\leq \int_{\mathcal M_\delta}  \sum_{i=1}^n \left[\left|\ln\left(\frac{u_i}{1-u_0}\right)\left(\frac{u_i}{1-u_0}\right)\right| + \left|\ln\left(1-u_0^\delta\right)\left(\frac{u_i}{1-u_0}\right) \right| 
	+ \frac{u_i}{1-u_0}\right] (u_0^\delta - u_0) dx.\\
	\end{align*}
	Using the fact that $\mathop{\max}\limits_{x\in [0,1]} |x\ln x| = \frac{1}{e}$, the fact that $u_0^\delta \geq \delta$ and  that $\frac{u_i}{1-u_0} \le 1$, we obtain that, if $\delta \leq 1/2$, 
	\begin{align*}
	&	\int_{\mathcal M_\delta}  \sum_{i=1}^n-[\ln u_i^\delta + 1](u_0^\delta - u_0)\frac{u_i}{1-u_0} dx \\
	&\le n \left(\frac{1}{e} + |\ln(1-\delta)| + 1\right) \int_{\mathcal M_\delta}(u_0^\delta - u_0) dx \\
	&\le n \left(\frac{1}{e} + |\ln 2| + 1\right) \int_{\mathcal M_\delta}(u_0^\delta - u_0) dx.\\
	\end{align*}
	In addition, it holds that 
	\begin{align*}
	\int_{\mathcal M_\delta} [\ln u_0^\delta + 1](u_0^\delta - u_0) dx=  (\ln \delta  + 1)\int_{\mathcal M_\delta}(u_0^\delta - u_0) dx.
	\end{align*}
	Finally, the last terms in \eqref{eq:terms_linear} are estimated as follows:
	\begin{align*}
	-\int_{\mathcal M_\delta}\sum_{i=1}^nf_i(u_i^\delta-u_i) dx & \le \sum_{i=1}^n  \|f_i\|_{L^\infty(\Omega)} \int_{\mathcal M_\delta}(u_0^\delta - u_0) dx \\
	& \leq \left(\sum_{i=1}^n  \|\bar w_i\|_{L^\infty(\Omega)} + 3n \beta\right) \int_{\mathcal M_\delta}(u_0^\delta - u_0) dx,\\
	\end{align*}
	using $|u_i^\delta - u_i| \leq u_0^\delta - u_0$. Combining all these estimates gives
	\begin{equation*}
	F_{\bar \bw}(\bu^\delta) - F_{\bar \bw}(\bu) \le (\ln \delta  + C)\int_{\mathcal M_\delta}(u_0^\delta - u_0) dx,
	\end{equation*}
	with 
	$$
	C = \sum_{i=1}^n  \|\bar w_i\|_{L^\infty(\Omega)} + 3n \beta +n\left(\frac{1}{e} + \ln 2 +1\right).
	$$
Finally, we observe that $\int_{\mathcal M_\delta} (u_0^\delta - u_0) >0$, because the function $u_0^\delta -u_0$ is assumed to be positive on the set $\mathcal M_\delta$ which has positive measure.
	Thus, if the value of $\delta$ is chosen so that $\delta< \min\left(\frac{1}{2}, e^{-C}\right)$, we have that $\ln \delta  + C <0$ which implies
	$$
	F_{\bar \bw}(\bu^\delta) - F_{\bar \bw}(\bu) <0,
	$$
the desired contradiction. We have thus proved that, for every minimizer $\bu \in \mathcal A$ to (\ref{eq:var_problem}), there exists $\delta_{\bar \bw}>0$ such that 
	$u_0 \geq \delta_{\bar \bw}$, where $u_0=1- \sum\limits_{i=1}^n u_i$.
	
	Moreover,  the value of $\delta$ can be chosen so that it only depends on $n$, $\beta$, $\Omega$, $\tau$, and $N$ as soon as $\bar\bw$ is assumed to satisfy $\|\bar \bw\|_{(H^2(\Omega))^n} \leq N$, since $H^2(\Omega)$ is compactly embedded in $L^\infty(\Omega)$. \\
	
\bfseries Step~2: \normalfont Let us now show that there exists $1 > \delta >0$, which only depends on $n$, $\Omega$, $\beta$, $\tau$, and $\sum\limits_{i=1}^n \|\bar w_i\|_{L^\infty(\Omega)}$,
such that  $1 - u_0 = \sum\limits_{n=1}^n u_i \geq \delta$ almost everywhere in $\Omega$. As before, the precise value of $\delta$ will be specified later in the proof. As in Step~1, we argue by contradiction assuming that the set 
	$$
	\mathcal N_\delta := \{ x \in \Omega \; : \; 1-\delta < u_0(x) \}
	$$
	has positive measure. Let us now define
	\begin{align*}
	u_0^\delta := \min(u_0,1-\delta),
	\end{align*}
	and $u_i^\delta$ as in \eqref{eq:u1u2delta}. We still obtain that $\bu^\delta:=(u_1^\delta, \dots, u_n^\delta)\in \mathcal A$ and that $u_0^\delta = 1- \sum\limits_{i=1}^n u_i^\delta$. Doing similar calculations as in Step~1 gives
	\begin{equation}
	\label{estim-F}
	\begin{split}
	 & F_{\bar \bw}(\bu^\delta) - F_{\bar \bw}(\bu) \\
	 & \leq \int_{\mathcal N_\delta} \sum_{i=1}^n[\ln u_i^\delta + 1](u_i^\delta - u_i)+ [\ln u_0^\delta + 1](u_0^\delta - u_0)dx - \int_{\mathcal N_\delta} \sum_{i=1}^nf_i(u_i^\delta-u_i) dx.
	\end{split}
	\end{equation}
On the one hand, it holds that
	\begin{align*}
	&\int_{\mathcal N_\delta} \sum_{i=1}^n [\ln u_i^\delta + 1](u_i^\delta - u_i) dx= \int_{\mathcal N_\delta}\sum_{i=1}^n-[\ln u_i^\delta + 1](u_0^\delta - u_0)\frac{u_i}{1-u_0} dx \\
	&\le \int_{\mathcal N_\delta}\sum_{i=1}^n \left [\left|\ln\left(\frac{u_i}{1-u_0}\right)\left(\frac{u_i}{1-u_0}\right)\right| + 1 \right](u_0 - u_0^\delta) dx+
	\int_{\mathcal N_\delta}\frac{\sum\limits_{i=1}^n u_i}{1-u_0}\left(-\ln\left(1-u_0^\delta\right)(u_0^\delta - u_0)\right)dx\\
	&\le \left( n\left(\frac{1}{e} +  1\right) + \ln \delta \right) \int_{\mathcal N_\delta}(u_0 - u_0^\delta) dx,
	\end{align*}
	as $\sum\limits_{i=1}^n u_i = 1-u_0$ and $1-u_0^\delta =\delta $ in $\mathcal N_\delta$. Furthermore, if $\frac{1}{2} \geq \delta$, we can estimate the second term on the right-hand side of \eqref{estim-F} as
	\begin{align*}
	\int_{\mathcal N_\delta} [\ln u_0^\delta + 1](u_0^\delta - u_0) dx \le (\ln 2  + 1)\int_{\mathcal N_\delta}(u_0 - u_0^\delta) dx,
	\end{align*}
while the terms involving $f_i$ are estimated with similar calculations as in Step~1. Then we have 
	\begin{align*} 
	&F_{\bar \bw}(\bu^\delta) - F_{\bar \bw}(\bu) \le (\ln \delta  + C)\int_{\mathcal N_\delta}(u_0 - u_0^\delta) dx,
	\end{align*}
	with 
	$$
	C = \ln 2  +1 + \sum_{i=1}^n\|w_i\|_{L^\infty(\Omega)} + 3n\beta +n\left(\frac{1}{e} + 1\right).
	$$
	We then reach the desired contradiction as soon as the value of $\delta$ is chosen such that $\delta < \min(1/2, e^{-C})$. Moreover, as in Step~1, 
	we obtain that the value of $\delta$ can be chosen so that it only depends on $n$, $\beta$, $\Omega$, $\tau$, and $N$, as soon as $\bar\bw$ is assumed to satisfy $\|\bar \bw\|_{(H^2(\Omega))^n} \leq N$.
	
	\medskip
	
	\bfseries Step~3: \normalfont It remains to prove that for all $1\leq i \leq n$  there exists $1> \delta >0$ such that $u_i\geq \delta$ a.e. in $\Omega$. Without any loss of generality  it is sufficient to prove the claim for $i=1$. To this end, let us again reason by contradiction  and assume that the set 
	$$
	\mathcal O_\delta:= \left\{ x\in \Omega:  \; u_1(x) < \delta \right\}
	$$
	has  positive measure. Denoting by $\delta^0$ the positive lower bound on $1-u_0$ obtained from Step~2, assuming that $0 < \delta \le  \frac{\delta^0}{2}$, we define
	\begin{align*}
	u_1^\delta &= \max( u_1, \delta), \\
	u_j^\delta &:= u_j - (u_1^\delta - u_1)\frac{u_j}{1-u_0-u_1},\quad j=2,\ldots, n,\\
	u_0^\delta & := u_0.\\
	\end{align*}
	Denoting by $\bu^\delta:=(u_1^\delta, \dots, u_n^\delta)$, we again have $\bu^\delta \in \mathcal A$ and that $u_0^\delta = u_0 = 1 - \sum\limits_{i=1}^n u_i^\delta$. 
	Arguing as in Steps~1 and 2 gives again the existence of a constant $c>0$, which only depends on $\tau$, $n$, $\beta$, and $\Omega$ such that, if $\delta < \min\left( 1/2, \delta^0/2\right)$,  then
	\begin{equation*} 
	F_{\bar \bw}(\bu^\delta) - F_{\bar \bw}(\bu) \le (C + \ln \delta )\int_{\mathcal O_\delta}(u_1^\delta - u_1) dx,
	\end{equation*}
	with 
	$$
	C:= \sum_{i=1}^n \|\bar w_i\|_{L^\infty(\Omega)} + c.
	$$
      Thus, we obtain that $F_{\bar \bw}(\bu^\delta) - F_{\bar \bw}(\bu) <0$ if the value of $\delta$ is chosen such that $\delta < \min\left( \frac{1}{2}, \delta^0/2, e^{-C}\right)$ which yields the desired contradiction. 
      Moreover, if $\bar \bw$ is assumed to satisfy $\|\bar \bw\|_{(H^2(\Omega))^n} \leq N$, the value of $\delta$ can be chosen such that it only depends on $\tau$, $n$, $\Omega$, $\beta$, and $N$. Hence the desired result.
\end{proof}
We remark that the technique of constructing competitors to the scalar Cahn-Hilliard energy was also used in \cite{Gelantalis2017}, yet in a different context.

\begin{lemma}\label{lem:weak}
Every minimizer $\bu \in \mathcal A$  of \eqref{eq:var_problem} belongs to $\mathcal B$ and is a weak solution to \eqref{discr2} in the sense that
\begin{align}\label{eq:w_to_u_weaknew}
\sum_{i=1}^n \into (\ln u_i- \ln u_0)\phi_i + \eps \nabla u_0\cdot \nabla \phi_0 dx = \sum_{i=1}^n \into \left( \bar w_i + \beta(1- 2 u_0^p)\right)\phi_i dx,
\end{align}
for all $\bphi  \in \mathcal B \cap (L^\infty(\Omega))^n$. In particular,  
\begin{equation}\label{eq:distrib}
 \ln u_i - \ln u_0 + \eps \Delta u_0 = \bar w_i + \beta (1-2u_0^p) \quad \mbox{ in }\mathcal  D'(\Omega), 
\end{equation}
for all $ 1\leq i \leq n$.
\end{lemma}

\begin{proof} 
Fix $\bphi  \in \mathcal B \cap (L^\infty(\Omega))^n $. Due to Lemma~\ref{lem:pos_minimizers}  we know that for every  $t>0$ sufficiently small  $\bu + t\bphi \in \mathcal A \cap \mathcal B$. Moreover, since $\bu$ is a minimizer of \eqref{eq:var_problem} it holds that 
	\begin{equation}
	\label{dis.3.6-1}
	\begin{split}
	0 &\le \lim_{t\to 0} \frac{F_{\bar \bw}(\bu + t\phi) - F_{\bar \bw}(\bu)}{t} \\
	&= \lim_{t\to 0}\into \sum_{i=0}^n \frac{(u_i+t\phi_i)\ln (u_i+t\phi_i) - u_i\ln u_i}{t} + \frac{\eps}{2}\frac{|\nabla(u_0+t\phi_0)|^2-|\nabla u_0|^2}{t} + \sum_{i=1}^n f_i \phi_i dx. 
	\end{split}
	\end{equation}
	 Lemma~\ref{lem:pos_minimizers} and the Lebesgue dominated convergence theorem then give
	\begin{equation}
	\label{dis.3.6-2}
	\begin{split}
	& \lim_{t\to 0}\into \sum_{i=0}^n \frac{(u_i+t\phi_i)\ln (u_i+t\phi_i) - u_i\ln u_i}{t} + \frac{\eps}{2}\frac{|\nabla(u_0+t\phi_0)|^2-|\nabla u_0|^2}{t} + \sum_{i=1}^n f_i \phi_i dx\\
	&=\into \sum_{i=1}^n (\ln u_i - \ln u_0)\phi_i + \eps \nabla u_0 \cdot \nabla \phi_0 + \sum_{i=1}^n f_i \phi_i dx.
	\end{split}
	\end{equation}
Combining \eqref{dis.3.6-1} and \eqref{dis.3.6-2} yields
	$$
	\into \sum_{i=1}^n (\ln u_i - \ln u_0)\phi_i + \eps \nabla u_0 \cdot \nabla \phi_0 + \sum_{i=1}^n f_i \phi_i dx \geq 0.
	$$
	Replacing $\bphi$ by $-\bphi$  we obtain that $\bu$ satisfies \eqref{eq:w_to_u_weaknew}.
	Finally, for all $1\leq i \leq n$, we obtain \eqref{eq:distrib} by considering a test function $\bphi = (\phi_j)_{1\leq j \leq n}$ such that $\phi_i \in \mathcal D(\Omega)$ and $\phi_j = 0$ for all $1\leq j\neq i \leq n$. 
\end{proof}


\begin{lemma} System~\eqref{eq:w_to_u_weaknew} has at most one solution $\bu \in \mathcal A \cap \mathcal B$.
\end{lemma}
\begin{proof}
Let us suppose that there exist two weak solutions $\bu$ and $\tilde\bu $ in $\mathcal A \cap \mathcal B$ to \eqref{eq:w_to_u_weaknew}. Subtracting the respective equations
yields
	\begin{equation*}
 	0 = \into \sum_{i=1}^n (\ln u_i - \ln \tilde u_i  - (\ln u_0  - \ln \tilde u_0 ))\phi_i - (\ln u_0  - \ln \tilde u_0 )\phi_0 + \eps \nabla (u_0 - \tilde u_0)\cdot \nabla \phi_0 dx,
	\end{equation*}
for all $\bphi  \in \mathcal B \cap (L^\infty(\Omega))^n $. Now, choosing $\phi_i = u_i - \tilde u_i$ for all $1\leq i \leq n$ so that 
$\phi_0 = - \sum\limits_{i=1}^n (u_i - \tilde u_i) = u_0 - \tilde u_0 $ we obtain 
	\begin{equation*}
	\begin{split}
	0 &= \into \sum_{i=1}^n (\ln u_i  - \ln \tilde u_i  - (\ln u_0  - \ln \tilde u_0 ))(u_i - \tilde u_i) + \eps \nabla (u_0 - \tilde u_0)\cdot \nabla ( u_0 -\tilde u_0) dx \\
	 &= \into \sum_{i=1}^n (\ln u_i  - \ln \tilde u_i )(u_i - \tilde u_i) + (\ln u_0  - \ln \tilde u_0 )(u_0 - \tilde u_0) + \eps |\nabla (u_0 - \tilde u_0)|^2 dx.\\
	\end{split}
	\end{equation*}
	The monotonicity of the logarithm implies $(\ln x -\ln y )(x-y) \ge 0$ for all  $x,\,y > 0$, which implies that all terms in the above integral are non-negative, so that a.e. in $\Omega$
	$$
	 (\ln u_i  - \ln \tilde u_i )(u_i - \tilde u_i) = 0  \quad \forall \, 0\leq i \leq n
	$$
	and 
	$$
	\nabla (u_0 - \tilde u_0) = 0.
	$$
	The strict monotonicity of the logarithm thus implies that $u_i = \tilde u_i$ for $0\leq i \leq n$, which yields the desired result. 
\end{proof}
We then let define $\mathcal S_2: (H^2(\Omega))^n \to \mathcal A$ as the application which to any $\bar \bw\in (H^2(\Omega))^n$ associates  the unique minimizer $\bu$ of \eqref{eq:var_problem}, 
which is also the unique solution in $\mathcal A \cap \mathcal B$ to \eqref{eq:w_to_u_weaknew}. Our next aim is to prove that $\mathcal S_2$ is a continuous map. To this end, we are going to prove that, if $\left( \bar \bw^{(m)}\right)_{m\in \mathbb{N}}$ is a sequence in $(H^2(\Omega))^n$ which strongly converges in $(H^2(\Omega))^n$ to some 
$\bar \bw \in (H^2(\Omega))^n$, then the sequence of minimizers to the functionals $\left( F_{\bar \bw^ {(m)}} \right)_{m\in \mathbb{N}}$ converges to the minimizer of $F_{\bar \bw}$. We first collect some regularity properties of the minimizers.
%
%
\begin{lemma}
	\label{lemma2}
	For all $\bar \bw \in (H^2(\Omega))^n$, it holds that $\mathcal S_2(\bar \bw) \in (H^2(\Omega))^n$. 
	Moreover, for all $N>0$, there exists a constant $M_1>0$, which only depends on $n$, $\Omega$, $\eps$, $\beta$, and $N$, such that for all $\bar \bw \in (H^2(\Omega))^n$ with $\|\bar \bw\|_{(H^2(\Omega))^n} \leq N$, we have
	$$
	\|\mathcal S_2(\bar \bw) \|_{(H^2(\Omega))^n} \leq M_1.
	$$	
\end{lemma}

\begin{proof}
	Let $\bar \bw \in (H^2(\Omega))^n$,  $\bu:= \mathcal S_2(\bar \bw)$. We first point out that, since $\bu$ is a minimizer of $F_{\bar \bw}$ on $\mathcal A$ and therefore $F_{\bar \bw}(\bu) \le F_{\bar \bw}(\mathbf 0) = 0$, it holds  
	\begin{align*}
	\frac{\eps}{2} \|\nabla u_0\|_{L^2(\Omega)}^2 & \leq n |\Omega| \sup_{x\in [0,1]} |x\ln x| + 3\beta |\Omega|\left( \sum_{i=1}^n \|\bar w_i\|_{L^\infty(\Omega)}\right),\\
	& \leq \frac{n |\Omega|}{e}+ 3\beta |\Omega|\left( \sum_{i=1}^n \|\bar w_i\|_{H^2(\Omega)}\right),\\
	& \leq \frac{n |\Omega|}{e} + 3\beta \sqrt{n}|\Omega| \|\bar \bw\|_{(H^2(\Omega))^n}.\\
	\end{align*}
Moreover, from Lemma~\ref{lem:weak} we have 
	\begin{equation}\label{eq:relation}
	\ln u_i - \ln u_0 + \eps \Delta u_0 = \bar w_i + \beta (1 - 2u_0^p)
	\end{equation}
for all $1\leq i \leq n$, in the sense of distributions. How Lemma~\ref{lem:pos_minimizers} implies $\|\ln u_i\|_{L^\infty(\Omega)} \leq |\ln \delta_{\bar w}|$ and $\|\ln u_0\|_{L^\infty(\Omega)} \leq |\ln \delta_{\bar w}|$. This yields  $\Delta u_0 \in L^2(\Omega)$ and  
	\[
	\begin{split}
	\eps \|\Delta u_0 \|_{L^2(\Omega)} & \leq |\Omega|^{1/2} \left( 2 |\ln \delta_{\bar \bw}|  + \|\bar w_i\|_{L^\infty(\Omega)} + 3 \beta\right)  \\
	& \leq |\Omega|^{1/2} \left( 2 |\ln \delta_{\bar \bw}|  + 
	C_e\|\bar \bw\|_{(H^2(\Omega))^n} + 3 \beta\right),
	\end{split}
	\]
	where $C_e$ is the embedding constant for $H^2(\Omega) \hookrightarrow L^\infty(\Omega)$. 	Moreover, if $\bar\bw$ satisfies $\|\bar \bw\|_{(H^2(\Omega))^n} \leq N$, there exists $\delta >0$, whose value only depends on $n$, $\Omega$, $\beta$, and $N$, such that 
	$\|\ln u_i\|_{L^\infty(\Omega)} \leq |\ln \delta|$ and $\|\ln u_0\|_{L^\infty(\Omega)} \leq |\ln \delta|$. Hence, in this case, 
	$$
	\eps \|\Delta u_0 \|_{L^2(\Omega)} \leq |\Omega|^{1/2} \left( 2 |\ln \delta| + N + 3 \beta\right).
	$$
	Let us now prove that  $\nabla u_i \in L^2(\Omega)$ for all $1\leq i \leq n$. Taking into account \eqref{eq:relation}  we obtain that 
	\[
	\bar w_i- \bar w_j= \ln \frac{u_i}{u_j} \quad \forall \, 1\leq i,j \leq n, 
	\]
	which implies that $u_i= u_j e^{\bar w_i- \bar w_j}$. Then, for all $1\leq i \leq n$, it holds that
	\[
	\begin{split}
	-\nabla u_0 = \nabla \left(\sum_{j=1}^n u_j\right) &= \nabla \l(u_i \left(1+ \sum_{1\leq j\neq i \leq n}e^{\bar w_i- \bar w_j}\right) \r)\\
	&=  \left(1+ \sum_{1\leq j\neq i \leq n}e^{\bar w_i- \bar w_j}\right) \nabla u_i+ u_i \sum_{1\leq j\neq i \leq n}e^{\bar w_i- \bar w_j}\nabla(\bar w_i- \bar w_j),
	\end{split}
	\]
	so that 
	$$
	\nabla u_i= \frac{-\nabla u_0+u_i \sum\limits_{1 \leq  j\neq i \leq n}e^{\bar w_i- \bar w_j}\nabla(\bar w_i- \bar w_j)}{1+ \sum\limits_{1\leq  j\neq i \leq n}e^{\bar w_i- \bar w_j} } = \frac{-\nabla u_0+\sum\limits_{1 \leq  j\neq i \leq n}u_j\nabla(\bar w_i- \bar w_j)}{1+ \sum\limits_{1\leq  j\neq i \leq n}e^{\bar w_i- \bar w_j}}.
	$$
	Thus, taking into account that $0 \le u_j \le 1$ for all $1\leq j \leq n$, we obtain    $\nabla u_i \in L^2(\Omega)$ for all $1\leq i \leq n$ and 
	\[
	\|\nabla u_i\|_{L^2(\Omega)} \le  \|\nabla u_0\|_{L^2(\Omega)}+ (n-1) \|\nabla \bar w_i\|_{L^2(\Omega)} + \sum_{1\leq j\neq i \leq n} \|\nabla \bar w_j\|_{L^2(\Omega)}.
	\]
	Moreover, using the fact that $d\leq 3$ yields the compact embedding $H^1(\Omega) \hookrightarrow L^4(\Omega)$, there exists a constant $C>0$ which only depends on $\Omega$ and $n$ such that 
	\begin{align*}
	 \|\nabla u_i\|_{L^4(\Omega)} & \le  \|\nabla u_0\|_{L^4(\Omega)}+ (n-1) \|\nabla \bar w_i\|_{L^4(\Omega)} + \sum_{1\leq j\neq i \leq n} \|\nabla \bar w_j\|_{L^4(\Omega)},\\
	 & \leq C \left(\|u_0\|_{H^2(\Omega)}+ \|\bar \bw\|_{H^2(\Omega)^n}\right).\\
	\end{align*}
	Finally, for all $1\leq i \leq n$ we have
	\begin{align*}
	 - \Delta u_0 & =  \divergenz \left( \left(1+ \sum_{1\leq j\neq i \leq n}e^{\bar w_i- \bar w_j}\right) \nabla u_i+ u_i \sum_{1\leq j\neq i \leq n}e^{\bar w_i- \bar w_j}\nabla(\bar w_i- \bar w_j)\right),\\
	 & = 2 \left( \sum_{1\leq j\neq i \leq n}e^{\bar w_i- \bar w_j}\nabla(\bar w_i- \bar w_j)\right) \cdot \nabla u_i + \left( 1+ \sum_{1\leq j\neq i \leq n}e^{\bar w_i- \bar w_j} \right) \Delta u_i\\
	 & \qquad+ u_i \divergenz \left( \sum_{1\leq j\neq i \leq n}e^{\bar w_i- \bar w_j}\nabla(\bar w_i- \bar w_j) \right)\\
	  & = 2 \sum_{1\leq j\neq i \leq n}e^{\bar w_i- \bar w_j}\nabla(\bar w_i- \bar w_j)\cdot \nabla u_i + \left( 1+ \sum_{1\leq j\neq i \leq n}e^{\bar w_i- \bar w_j} \right) \Delta u_i\\
	 & \qquad+ u_i \sum_{1\leq j\neq i \leq n}e^{\bar w_i- \bar w_j}|\nabla(\bar w_i- \bar w_j)|^2 + e^{\bar w_i- \bar w_j}\Delta(\bar w_i- \bar w_j).\\
	\end{align*}
	Hence, 
	\begin{align*}
	 \Delta u_i & = \frac{1}{\left( 1+ \sum\limits_{1\leq j\neq i \leq n}e^{\bar w_i- \bar w_j} \right) } \left[ - \Delta u_0 -2 \sum_{1\leq j\neq i \leq n}e^{\bar w_i- \bar w_j}\nabla(\bar w_i- \bar w_j)\cdot \nabla u_i\right]\\
	 & \qquad + \frac{1}{\left( 1+ \sum\limits_{1\leq j\neq i \leq n}e^{\bar w_i- \bar w_j} \right) }  \left[  -  \sum_{1\leq j\neq i \leq n}u_j|\nabla(\bar w_i- \bar w_j)|^2 - u_j\Delta(\bar w_i- \bar w_j)\right],\\
	\end{align*}
	which implies that $\Delta u_i\in L^2(\Omega)$ and 
	\begin{align*}
	 \|\Delta u_i\|_{L^2(\Omega)} & \leq  \|\Delta u_0\|_{L^2(\Omega)} + 2 \|\nabla u_i\|_{L^4(\Omega)} \|\nabla (\bar w_i- \bar w_j)\|_{L^4(\Omega)}\\
 	 & \qquad+ \sum_{1\leq j\neq i \leq n}\left( \|\nabla (\bar w_i- \bar w_j)\|_{L^4(\Omega)} \right)^2 + \|\Delta (\bar w_i -\bar w_j) \|_{L^2(\Omega)}.
	\end{align*}
	Thus, there exists a constant $C'>0$, which only depends on $\Omega$ and $n$, such that 
	$$
	\|\Delta u_i\|_{L^2(\Omega)} \leq  C'\left( \|u_0\|_{H^2(\Omega)} + \|u_0\|_{H^2(\Omega)}^2 + \|\bar \bw\|_{(H^2(\Omega))^n} + \|\bar \bw\|_{(H^2(\Omega))^n}^2\right).
	$$
	Collecting all these estimates gives the desired result. 
\end{proof}

\begin{lemma}\label{contS2}
The map $\mathcal S_2: (H^2(\Omega))^n \to \mathcal A \subset (L^\infty(\Omega))^n$ is continuous. 
\end{lemma}

\begin{proof}
 The continuity of the map $\mathcal S_2$ is a consequence of the bounds of Lemma~\ref{lemma2}. Indeed, let $\left( \bar \bw^{(m)} \right)_{m\in \mathbb{N}} \subset (H^2(\Omega))^n$ be a sequence  strongly converging to some 
 $\bar \bw \in (H^2(\Omega))^n$. Set $\bu^{(m)}:= \mathcal S_2(\bar \bw^{(m)})$ for all $m\in \mathbb{N}$. 
 Let us prove that $\left( \bu^{(m)}\right)_{m\in \mathbb{N}}$ strongly converges in $(L^\infty(\Omega))^n$ to $\bu:= \mathcal S_2(\bar\bw)$.  
 
 First of all, since the sequence $\left( \bar \bw^{(m)} \right)_{m\in \mathbb{N}}$ is bounded in $(H^2(\Omega))^n$, then Lemma~\ref{lemma2} entails that also the sequence $\left( \bu^{(m)}\right)_{m\in \mathbb{N}}$ is bounded in $(H^2(\Omega))^n$. 
 Up to the extraction of a subsequence (still denoted by  $\left( \bu^{(m)}\right)_{m\in \mathbb{N}}$ for the sake of simplicity), there exists $\tilde \bu \in (H^2(\Omega))^n$ such that 
 $$
 \bu^{(m)} \mathop{\rightharpoonup} \tilde \bu \mbox{ weakly in } (H^2(\Omega))^n \quad \text{ as }m \to +\infty.
 $$
 Let us prove that  necessarily  $\tilde \bu = \bu$, which will imply that the whole sequence $\left(\bu^{(m)}\right)_{m\in \mathbb{N}}$ weakly converges in $(H^2(\Omega))^n$ to $\bu$. 
 
Let us observe that the compact embeddings $ (H^2(\Omega))^n\hookrightarrow (L^\infty(\Omega))^n$ and $H^2(\Omega))^n\hookrightarrow(H^1(\Omega))^n $ imply  
 $$
  \bu^{(m)} \mathop{\to}  \tilde \bu \mbox{ strongly in }(L^\infty(\Omega))^n \quad \mbox{ and } \quad   u_0^{(m)} \mathop{\to}  \tilde u_0 \mbox{ strongly in }H^1(\Omega)
 $$
 as ${m\to +\infty}$.
 Thus, we obtain  
 $$
 F_{\bar \bw^{(m)}}\left( \bu^{(m)} \right) \mathop{\to}   F_{\bar \bw}\left( \tilde \bu \right) \quad \text{ as } {m\to +\infty}. 
 $$
 Besides, since for all $\bv \in \mathcal A$, $ F_{\bar \bw^{(m)}}\left( \bu^{(m)} \right)  \leq  F_{\bar \bw^{(m)}}\left( \bv \right)$ and
 $\displaystyle  F_{\bar \bw^{(m)}}(\bv)  \mathop{\to}   F_{\bar \bw}(\bv)$ as $m \to +\infty$, we  obtain  
 $$
F_{\bar \bw}\left( \tilde \bu \right) \leq  F_{\bar \bw}(\bv) \quad \forall \, \bv \in \mathcal A.
 $$
 Hence, $\tilde \bu$ is the unique minimizer of $F_{\bar \bw}$ on $\mathcal A$, i.e., $\tilde \bu = \bu$. As a consequence, the whole sequence $\left( \bu^{(m)}\right)_{m\in \mathbb{N}}$ weakly converges to $\bu$ in $(H^2(\Omega))^n$. Finally, 
 the compact embedding $(H^2(\Omega))^n \hookrightarrow (L^\infty(\Omega))^n $ implies that the sequence $\left( \bu^{(m)}\right)_{m\in \mathbb{N}}$ strongly converges in $(L^\infty(\Omega))^n$ to $\bu$, which yields the desired convergence. 
 Hence, the continuity of the map $\mathcal S_2$. 
\end{proof}

\subsection{Proof of Theorem~\ref{thm:existence_time_discrete}}\label{sec:proof1}

\begin{proof}[Proof of Theorem~\ref{thm:existence_time_discrete}]
 Let us define $\mathcal S: \mathcal A \to \mathcal A$ as $\mathcal S= \mathcal S_2 \circ \mathcal S_1$, with $\mathcal S_1: \mathcal A \to (H^2(\Omega))^n$ defined in Section~\ref{subsec:ex_w} and $\mathcal S_2: (H^2(\Omega))^n \to \mathcal A$ defined in Section~\ref{subsec:u_from_w}. 
 Thanks to Lemma~\ref{lem:S1cont} and Lemma~\ref{contS2} we obtain that $\mathcal S$ is continuous. Besides, using Lemma~\ref{thm3} together with Lemma~\ref{lemma2} 
 we obtain that $\mathcal S(\mathcal A)$ is a bounded subset of $(H^2(\Omega))^n$ and hence a relatively compact subset of $(L^\infty(\Omega))^n$. 
 Since $\mathcal A$ is a closed convex non-empty subset of $(L^\infty(\Omega))^n$, Schauder's fixed point theorem ensures the existence of a fixed point $\bu^{p+1} \in \mathcal A$ such that $\bu^{p+1} = \mathcal S(\bu^{p+1})$. 
 Gathering the different results proved in Section~\ref{subsec:ex_w} and Section~\ref{subsec:u_from_w} yield the desired properties on the fixed-point $\bu^{p+1}$. 
\end{proof}

\section{Estimates on the solutions of the time discrete regularized system}\label{apriori-sect}

Let $T>0$ be a fixed final time. 
For all $0< \tau \le 1$ Theorem~\ref{thm:existence_time_discrete} implies that, for any initial condition $\bu^0 \in \mathcal A \cap (H^2(\Omega))^n$, there exists a sequence $(\bu^p)_{p\in  \mathbb{N}} \subset \mathcal A \cap (H^2(\Omega))^n$  
defined by recursion  such that 
$(\bu^{p+1}, \bar \bw^{p+1})\in (\mathcal A \cap (H^2(\Omega))^n) \times (H^2(\Omega))^n$ is a solution to \eqref{dis-sys1}-\eqref{eq:w_to_u_weak0} for all $p\in \mathbb{N}$. 

For all $p\in \mathbb{N}^*$  if $\bu^p:= (u_1^p, \dots, u_n^p)$ and $\bar \bw^p:= (\bar w_1^p, \dots, \bar w_n^p)$, we set 
	\[
	u_0^p:= 1 -\sum_{i=1}^n u_i^p, \quad u^p:=(u_0^p, u_1^p, \dots, u_n^p), \quad \text{and }  w_i^{p}:=  \ln u_i^p - \ln u_0^p \quad \forall \, 1\leq i \leq n. 
	\]
We also denote by $w_0^{p+1/2}:= -\eps \Delta u_0^{p+1} + \beta (1-2u_0^p)$ and finally set  
$w^{p+1}:=(w_0^{p+1/2}, w_1^{p+1}, \dots , w_n^{p+1})$, for all $p\in \mathbb{N}$.

We then define several piecewise constant in time functions as follows: for all $p\in \mathbb{N}^*$, for all $1\leq i \leq n$ and all $t \in (t_{p-1}, t_p]$, we set
\begin{equation}
\label{setting}
\begin{split}
\bu^{(\tau)}(t)  &= \bu^p,\quad u^{(\tau)}(t)  = u^p, \quad u_i^{(\tau)}(t)  = u_i^p,\quad u_0^{(\tau)}(t)  = u_0^p,\\
\bar \bw^{(\tau)} &= \bar \bw^p, \quad \bar w_i^{(\tau)} = \bar w_i^p, w_i^{(\tau)}   = w_i^p = \ln u_i^p - \ln u_0^p, \quad w_0^{(\tau)}  = w_0^{p-1/2}, \quad w^{(\tau)}  = w^p.
\end{split}
\end{equation}
At time $t=0$  we  define $\bu^{(\tau)}(0) =\bu^0$. Let $P^{(\tau)}\in \mathbb{N}^*$ be the lowest integer such that $t_{P^{(\tau)}}\geq T$. 
Furthermore, we introduce the time-shifted solution $\sigma_\tau \bu^{(\tau)}$ as
$$
\sigma_\tau \bu^{(\tau)}(t) = \bu^{p-1} \mbox{ for all } t \in (t_{p-1}, t_p],  \, p\in \mathbb{N}^*,
$$
whose components are given by  $(\sigma_\tau u_1^{(\tau)}, \dots, \sigma_\tau u_n^{(\tau)})$, and set  $\sigma_\tau u_0^{(\tau)}:= 1 - \sum\limits_{i=1}^n \sigma_\tau u_i^{(\tau)}$. For all $\bu = (u_0, u_1, \dots, u_n) \in (L^\infty(\Omega)\cap H^1(\Omega))\times (L^\infty(\Omega))^{n}$  we define
	\[
	\begin{split}
	E_{\text{conv}}(\bu)&= \into \sum_{i=0}^n u_i\ln u_i+ \frac{\eps}{2} |\nabla u_0|^2 dx \\
	\text{and} \quad E_{\text{conc}}(\bu)&= \into \beta u_0(1- u_0) dx.
	\end{split}
	\]
	\medskip
For all $\tau >0$ and $t>0$  we define the entropy functional
	\[
	\begin{split}
	E^{(\tau)}(t)&:= \into u_i^{(\tau)}(t) \ln u_i^{(\tau)}(t)+ \frac{\eps}{2} |\nabla u_0^{(\tau)}(t)|^2+ \beta \sigma_\tau u_0^{(\tau)}(t)(1- \sigma_\tau u_0^{(\tau)}(t)) dx \\
	&= E_{\text{conv}}(\bu^{(\tau)}(t))+ E_{\text{conc}}(\sigma_\tau \bu^{(\tau)}(t))
	\end{split}
	\] 
so that, for all $p\in \mathbb{N}$, 
	\begin{equation*}
	\begin{split}
	E^{(\tau)}(t_{p+1})&= \into \sum_{i=0}^n u_i^{p+1} \ln u_i^{p+1}+ \frac{\eps}{2} |\nabla u_0^{p+1}|^2+ \beta u_0^{p}(1- u_0^{p}) dx \\
	&= E_{\text{conv}}(\bu^{p+1}) + E_{\text{conc}}(\bu^{p}). 
	\end{split}
	\end{equation*}
	
\begin{remark}
\label{rem1.1}

It is easy to check that there exists a constant $\widetilde{C}>0$, independent of $\tau$, such that $E\left(\bu^{(\tau)}(t)\right) \geq - \widetilde{C}$ for all $t>0$. 

\end{remark}
The objective of this section is to collect some estimates on the solution $\bu^{(\tau)}$ which will be used in the sequel to pass to the limit as $\tau \to 0^+$ in the time discrete regularized system. We begin by stating an important property of the mobility matrix $M$ which will be used in the following. 
\begin{lemma}
\label{rem1}

Let $\bz\in \mathbb{R}^{n+1}$ and let $M$ be defined as in \eqref{mobility}. Then, for all $\bu\in \mathbb{R}_+^{n+1}$, 
	\[
	\bz^T M(\bu) \bz \ge 0.
	\]
\end{lemma}
	
\begin{proof}
Indeed, for all $\bu=(u_0, \dots, u_n)\in \mathbb{R}_+^{n+1}$ and all $\bz = (z_0, \dots, z_n)\in \mathbb{R}^{n+1}$  we have
	\begin{equation*}
	\begin{split}
	 \bz^T M( \bu) \bz & = \sum_{\overset{i,j=0}{j \ne i}}^n  z_i M_{ij}(\bu) z_j + \sum_{i=0}^n M_{ii}(\bu) z_{i}^2\\
	 & = \sum_{\overset{i,j=0}{j \ne i}}^n (-K_{ij}u_iu_j z_i z_j) + \frac{1}{2}\sum_{i=0}^n M_{ii}(\bu) z_{i}^2  + \frac{1}{2}\sum_{j=0}^n M_{jj}(\bu) z_{j}^2\\
	 & =  \sum_{\overset{i,j=0}{j \ne i}}^n (-K_{ij}u_iu_j z_i z_j) + \frac{1}{2} \sum_{\overset{i,j=0}{j \ne i}}^n K_{ij}u_iu_j z_{i}^2+ \frac{1}{2} \sum_{\overset{i,j=0}{i \ne j}}^n K_{ij}u_iu_j z_{j}^2\\
	  & =  \sum_{\overset{i,j=0}{j \ne i}}^n K_{ij}u_iu_j \l(\frac{1}{2}z_i^2 + \frac{1}{2}z_j^2 - z_i z_j\r)\\
	    & = \frac{1}{2} \sum_{\overset{i,j=0}{j \ne i}}^n K_{ij}u_iu_j (z_i - z_j)^2 \ge 0,\\
	\end{split}
	\end{equation*}
which gives us the conclusion.
\end{proof} 
We now state the monotonicity of the energy functional $E^{(\tau)}$.
\begin{lemma}
\label{lemma1}
For all $\tau>0$, the sequence $(E^{(\tau)}(t_p))_{p \in \N^*}$ is non-increasing. Moreover, there exists $C>0$ such that for all $\tau>0$ and all $t>0$, 
	\begin{equation}
	\label{grad1}
	\into |\nabla u_0^{(\tau)}(t)|^2 dx \le C.
	\end{equation}
	
\end{lemma}

\begin{proof}

We test each equation in \eqref{dis-sys1} with the test function $\phi_i= \bar w_i^{p+1} = w_i^{p+1}- w_0^{p+1/2} $ and then sum over $i=1,\ldots, n$. On the left-hand side, exploiting the convexity and concavity properties of
the functions $[0,1]\ni x \mapsto x\ln x$ and $[0,1]\ni x \mapsto x(1-x)$, respectively, together with Remark~\ref{rem:CL}, we have
	\begin{equation}
	\label{en-dis1}
	\begin{split}
	&\into \frac{u_i^{p+1}- u_i^p}{\tau}(w_i^{p+1}- w_0^{p+1/2}) dx \\
	&= \sum_{i=1}^n \into \frac{u_i^{p+1}- u_i^p}{\tau} \l[\ln u_i^{p+1} -\ln u_0^{p+1} +\eps \Delta u_0^{p+1} - \beta (1-2u_0^p) \r] dx \\
	&= \sum_{i=1}^n \into \frac{u_i^{p+1}- u_i^p}{\tau} \ln u_i^{p+1} dx+ \into \frac{u_0^{p+1}- u_0^p}{\tau} \left( \ln u_0^{(p+1} - \eps \Delta u_0^{p+1} + \beta (1- 2u_0^p)\right) dx \\
	& \ge \frac{1}{\tau} \l[E_{\text{conv}}({\bu}^{p+1})- E_{\text{conv}}({\bu}^{p})+ E_{\text{conc}}({\bu}^{p+1})- E_{\text{conc}}({\bu}^{p}) \r] \\
	&= \frac{1}{\tau} \l[E({\bu}^{p+1})- E({\bu}^p) \r].
	\end{split}
	\end{equation}
On the right-hand side, exploiting Lemma \ref{rem1} and the definition of the matrix $M$, see \eqref{mobility}, we have
	\begin{equation}
	\label{en-dis2}
	\begin{split}
	&\sum_{i=1}^n \into \biggl(\sum_{1 \le j \ne i \le n} K_{ij} u_i^{p+1} u_j^{p+1} \nabla (w_i^{p+1}- w_j^{p+1}) \\
	& \qquad + K_{i0} u_i^{p+1} u_0^{p+1} \nabla (w_i^{p+1}- w_0^{p+1/2}) \biggr) \cdot \nabla(w_i^{p+1}- w_0^{p+1/2}) dx \\
	&= \into (\nabla {w}^{p+1})^T M(\bu^{p+1}) \nabla { w}^{p+1} dx \le 0.
	\end{split}
	\end{equation}
From \eqref{en-dis1}-\eqref{en-dis2} it follows that
	\[
	\frac{1}{\tau} \l[E({\bu}^{p+1})- E({\bu}^p)\r] \le 0 \quad \forall \, p \in \N,
	\]
which implies that the sequence $(E({\bu}^p))_{p \in \N}$ is non-increasing. In particular, there exists a constant $C>0$ such that $E({ u}^p) \le E({u}^0) \le C$, which in turn entails
	\[
	\frac{\eps}{2} \into |\nabla u_0^p|^2 dx \le C,  
	\]
for every $p \in \N$. Taking into account \eqref{setting}  yields the desired result. 
\end{proof}

We now use the monotonicity of the entropy functional in order to establish some a-priori estimates that will be used to pass to the limit as $\tau \to 0^+$ in the time discrete system.

\begin{lemma}
\label{thm1}
There exists a constant $C>0$, independent of $\tau >0$, such that  
	\begin{align}
	\sum_{i=0}^n \int_0^T  \into \frac{|\nabla u_i^{(\tau)}|^2}{u_i^{(\tau)}} dx dt &\le C,  \nonumber \\ 
	\int_0^T \into |\Delta u_0^{(\tau)}|^2 dx dt & \le C, \nonumber \\ 
	\int_0^T \into (1- u_0^{(\tau)}) u_0^{(\tau)} |\nabla w_0^{(\tau)}|^2 dx dt & \le C, \label{dis.d} \\
	\tau \sum_{i=1}^n \int_0^T  \|w_i^{(\tau)}- w_0^{(\tau)}\|_{H^2(\Om)}^2 dt &\le C. \label{dis.b}
	\end{align}
\end{lemma}

\begin{proof}
First of all, let us introduce $\displaystyle k := \min_{0 \le i \ne j \le n} K_{ij}$. 
We test each equation of \eqref{dis-sys1} with $ \phi_i= \bar w_i^{p+1} = w_i^{p+1}- w_0^{p+1/2}$ and sum for $i=1,\ldots, n$. 
On the right-hand side  we have
	\begin{equation}
	\label{est1}
	\begin{split}
	\mathcal A:=& -\sum_{i=1}^n \into \biggl(\sum_{1 \le j \ne i \le n} K_{ij} u_i^{p+1} u_j^{p+1} \nabla (w_i^{p+1}- w_j^{p+1}) \\
	& \qquad + K_{i0} u_i^{p+1} u_0^{p+1} \nabla (w_i^{p+1}- w_0^{p+1/2}) \biggr) \cdot \nabla (w_i^{p+1}- w_0^{p+1/2}) dx \\
	& \qquad -\tau \Vert w_i^{p+1}- w_0^{p+1/2} \Vert_{H^2(\Omega)}^2 \\
	&= -\sum_{i=1}^n \into \biggl(\sum_{1 \le j \ne i \le n} (K_{ij}- k) u_i^{p+1} u_j^{p+1} \nabla (w_i^{p+1}- w_j^{p+1}) \\
	& \qquad + (K_{i0}- k) u_i^{p+1} u_0^{p+1} \nabla (w_i^{p+1}- w_0^{p+1/2}) \biggr) \cdot \nabla (w_i^{p+1}- w_0^{p+1/2})dx \\ 
	& \qquad-k \sum_{i=1}^n \into \biggl(\sum_{1 \le j \ne i \le n} u_i^{p+1} u_j^{p+1} \nabla (w_i^{p+1}- w_j^{p+1}) \\
	& \qquad + u_i^{p+1} u_0^{p+1} \nabla (w_i^{p+1}- w_0^{p+1/2}) \biggr) \cdot \nabla (w_i^{p+1}- w_0^{p+1/2})dx \\
	& \qquad -\tau \sum_{i=1}^n \Vert w_i^{p+1}- w_0^{p+1/2} \Vert_{H^2(\Omega)}^2 \\
	&= A+ B+ C.
	\end{split}
	\end{equation}
First of all observe that
	\[
	A= -\into (\nabla {\bw}^{p+1})^t \tilde M({\bu}^{p+1}) \nabla  {\bw}^{p+1} dx  \le 0,
	\]
where $\tilde M$ is the matrix defined as in \eqref{mobility} but with $K_{ij}$ replaced by $K_{ij}- k$  and where we used again Lemma \ref{rem1}. 
Let us consider the second term in \eqref{est1}. We have
	\begin{equation*}
	\begin{split}
	B&= -k \sum_{i=1}^n \into \biggl(\sum_{1 \le j \ne i \le n} u_i^{p+1} u_j^{p+1} \nabla (w_i^{p+1}- w_j^{p+1}) \\
	& \quad + u_i^{p+1} u_0^{p+1} \nabla (w_i^{p+1}- w_0^{p+1/2}) \biggr) \cdot \nabla (w_i^{p+1}- w_0^{p+1/2}) dx \\
	&= -k \sum_{i=1}^n \into \biggl(\sum_{1 \le j \ne i \le n}  u_i^{p+1} u_j^{p+1} \nabla(w_i^{p+1}- w_j^{p+1})+ u_i^{p+1} u_0^{p+1} \nabla w_i^{p+1}\biggr) \cdot \nabla w_i^{p+1} dx \\
	& \quad +k \sum_{i=1}^n \into \sum_{1 \le j \ne i \le n}  u_i^{p+1} u_j^{p+1} \nabla(w_i^{p+1}- w_j^{p+1}) \cdot \nabla w_0^{p+1/2} dx \\
	& \quad +2k \sum_{i=1}^n \into u_i^{p+1} u_0^{p+1} \nabla w_i^{p+1} \cdot \nabla w_0^{p+1/2} dx- k \sum_{i=1}^n \into u_i^{p+1} u_0^{p+1} |\nabla w_0^{p+1/2}|^2 dx \\
	&=: B_1+ B_2+ B_3+ B_4.
	\end{split}
	\end{equation*}
We estimate the terms of the expression above separately. First of all, taking into account \eqref{wi}  we have
	\begin{equation*}
	\begin{split}
	 B_1 &= -k \sum_{i=1}^n \into \l(\sum_{1 \le j \ne i \le n} u_i^{p+1} u_j^{p+1} \nabla (w_i^{p+1}- w_j^{p+1})+ u_i^{p+1} u_0^{p+1} \nabla w_i^{p+1} \r) \cdot \nabla w_i^{p+1} dx \\
	&= -k \sum_{i=1}^n  \int_\Omega  \sum_{0\leq j \neq i \leq n} u_i^{p+1} u_j^{p+1} \left( \nabla \ln u_i^{p+1} - \nabla \ln u_j^{p+1} \right) \cdot \nabla \left(\ln u_i^{p+1} - \ln u_0^{p+1}\r) dx\\
 	& = -k \sum_{i=1}^n \int_\Omega \sum_{0\leq j \neq i \leq n} (u_j^{p+1} \nabla u_i^{p+1} - u_i^{p+1} \nabla u_j^{p+1}) \cdot \left( \frac{\nabla u_i^{p+1}}{u_i^{p+1}} - \frac{\nabla u_0^{p+1}}{u_0^{p+1}}\right) dx \\
	& = -k \sum_{i=1}^n \int_\Omega ((1-u_i^{p+1})\nabla u_i^{p+1} - u_i^{p+1} \nabla (1-u_i^{p+1})) \cdot \left( \frac{\nabla u_i^{p+1}}{u_i^{p+1}} - \frac{\nabla u_0^{p+1}}{u_0^{p+1}}\right) dx \\
 	& = -k \sum_{i=1}^n \int_\Omega (\nabla u_i^{p+1} -u_i^{p+1} \nabla u_i^{p+1} + u_i^{p+1} \nabla u_i^{p+1}) \cdot \left( \frac{\nabla u_i^{p+1}}{u_i^{p+1}} - \frac{\nabla u_0^{p+1}}{u_0^{p+1}}\right) dx \\
 	& = -k \sum_{i=1}^n \int_\Omega \frac{|\nabla u_i^{p+1}|^2}{u_i^{p+1}} - \frac{\nabla u_i^{p+1} \cdot \nabla u_0^{p+1}}{u_0^{p+1}} dx \\
	&= - k \int_\Omega \sum_{i=1}^n  \frac{|\nabla u_i^{p+1}|^2}{u_i^{p+1}} - \frac{\nabla (1-u_0^{p+1})\cdot \nabla u_0^{p+1}}{u_0^{p+1}} dx \\
&= -k \into \sum_{i=1}^n \frac{\vert \nabla u_i^{p+1} \vert^2}{u_i^{p+1}}+ \frac{\vert \nabla u_0^{p+1} \vert^2}{u_0^{p+1}} dx \\
& =  - k \int_\Omega \sum_{i=0}^n  \frac{|\nabla u_i^{p+1}|^2}{u_i^{p+1}} dx. 
	\end{split}
	\end{equation*}
Exploiting symmetries gives
	\begin{equation*}
	B_2= k \sum_{i=1}^n \into \sum_{1 \le j \ne i \le n}  u_i^{p+1} u_j^{p+1} \nabla(w_i^{p+1}- w_j^{p+1}) \cdot \nabla w_0^{p+1/2} dx= 0.
	\end{equation*}
Moreover, taking into account the definition of $w_0^{p+1/2}$, see \eqref{w0p}, and using again the fact that $\nabla u_0^{p+1}\cdot {\bf n} = 0$ on $\partial \Omega$, we have
	\begin{equation*}
	\begin{split}
	B_3 &= 2k \sum_{i=1}^n \int_\Omega  u_i^{p+1} u_0^{p+1} \nabla w_0^{p+1/2} \cdot \nabla w_i^{p+1} dx \\
	& = 2k \sum_{i=1}^n \int_\Omega  u_i^{p+1} u_0^{p+1} \nabla w_0^{p+1/2} \cdot \left( \frac{\nabla u_i^{p+1}}{u_i^{p+1}} - \frac{\nabla u_0^{p+1}}{u_0^{p+1}}\right) dx\\
	&  = 2k \sum_{i=1}^n \int_\Omega \left(u_0^{p+1} \nabla u_i^{p+1} - u_i^{p+1} \nabla u_0^{p+1}\right)\cdot \nabla w_0^{p+1/2} dx\\
	&  =  2k \int_\Omega \left(u_0^{p+1} \nabla (1-u_0^{p+1}) - (1-u_0^{p+1}) \nabla u_0^{p+1}\right)\cdot \nabla w_0^{p+1/2} dx\\
	&  =  -2k \int_\Omega \nabla u_0^{p+1} \cdot \nabla w_0^{p+1/2} dx,\\
	&= 4 k\beta \int_\Omega \nabla u_0^{p+1} \cdot \nabla u_0^p dx - 2k \int_\Omega \nabla u_0^{p+1} \cdot \nabla (-\eps \Delta u_0^{p+1}) dx\\
	& \le 4 k\beta \|\nabla u_0^{p+1}\|_{L^2(\Om)} \|\nabla u_0^p\|_{L^2(\Om)} - 2k\eps \int_\Omega |\Delta u_0^{p+1}|^2dx \\
	& \le C- 2k\eps \int_\Omega |\Delta u_0^{p+1}|^2dx,
	\end{split}
	\end{equation*}
where in the last two passages we applied the Cauchy-Schwarz's inequality and then \eqref{grad1}. Finally, by using the constraint \eqref{constr} we get
	\begin{equation}
	\label{est1.5}
	B_4= -k \sum_{i=1}^n \into u_i^{p+1} u_0^{p+1} |\nabla w_0^{p+1/2}|^2 dx= -k \into (1- u_0^{p+1}) u_0^{p+1} |\nabla w_0^{p+1/2}|^2 dx.
	\end{equation}
From \eqref{est1}-\eqref{est1.5} we then have
	\[
	\begin{split}
	\mathcal A& \le -k \sum_{i=0}^n \into \frac{|\nabla u_i^{p+1}|^2}{u_i^{p+1}} dx+ C- 2k\eps \into |\Delta u_0^{p+1}|^2 dx \\
	& \qquad -k \into (1- u_0^{p+1}) u_0^{p+1} |\nabla w_0^{p+1/2}|^2 dx.
	\end{split}
	\]
Therefore, reasoning as in the proof of Lemma \ref{lemma1} gives
	\[
	\begin{split}
	&k \sum_{i=0}^n \into \frac{|\nabla u_i^{p+1}|^2}{u_i^{p+1}} dx+ 2k\eps \into |\Delta u_0^{p+1}|^2 dx+ k \into (1- u_0^{p+1}) u_0^{p+1} |\nabla w_0^{p+1/2}|^2 dx\\
	& \quad + \tau \sum_{i=1}^n \|w_i^{p+1}- w_0^{p+1/2}\|_{H^2(\Om)}^2 \\
	& \le C+ \frac{1}{\tau}\left( E({\bu}^p) -  E({\bu}^{p+1}\right) ).\\
	\end{split}
	\]
Multiplying this inequality by $\tau$, summing for $0\leq p \leq P^{(\tau)}-1$, and then using Remark \ref{rem1.1} yields 
\begin{align*}
 & k \sum_{i=1}^n \int_0^T \into \frac{|\nabla u_i^{(\tau)}|^2}{u_i^{(\tau)}} dx dt + 2k\eps \int_0^T \into |\Delta u_0^{(\tau)}|^2 dx dt + k \int_0^T \into (1- u_0^{(\tau)}) u_0^{(\tau)} |\nabla w_0^{(\tau)}|^2 dx \\
 & \qquad + \tau \int_0^T \sum_{i=1}^n \|w_i^{(\tau)}- w_0^{(\tau)}\|_{H^2(\Om)}^2 dt \\
 & \leq C(T+1) + E({\bu}^0) + \widetilde{C},
\end{align*}
which gives the desired result.
\end{proof}

\begin{remark}
\label{rem2}
From \eqref{dis.b}  we have in particular that $\l(\sqrt \tau (w_i^{(\tau)}- w_0^{(\tau)})\r)_{\tau >0}$ is uniformly bounded in $L^2((0,T); H^2(\Om))$. 
\end{remark}
Using similar arguments as in Lemma~\ref{thm1}, we can obtain further estimates. More precisely, we have the following result.
\begin{theorem}
\label{thm2}
There exists a constant $C>0$, independent of $\tau>0$, such that 
	\begin{equation*}
	\int_0^T \into u_i^{(\tau)} u_0^{(\tau)} |\nabla ( w_i^{(\tau)}- w_0^{(\tau)})|^2 dx dt \le C \quad \text{ for all } 1\leq i \leq n.
	\end{equation*}
\end{theorem}

\begin{proof}

We argue as in the proof of Lemma~\ref{thm1}. First of all, we test each equation in \eqref{dis-sys2} with $\phi_i=  w_i^{p+1}- w_0^{p+1/2}$ and sum for $i =1,\ldots, n$. On the right-hand side  we have
	\begin{equation*}
	\begin{split}
	& -\sum_{i=1}^n \into \biggl(\sum_{1 \le j \ne i \le n} K_{ij} u_i^{p+1} u_j^{p+1} \nabla (w_i^{p+1}-  w_j^{p+1})\\
	& \qquad \qquad + K_{i0} u_i^{p+1} u_0^{p+1} \nabla ( w_i^{p+1}- w_0^{p+1/2}) \biggr) \cdot \nabla ( w_i^{p+1}- w_0^{p+1/2})  dx \\
	& \quad - \sum_{i=1}^n \tau \| w_i^{p+1}- w_0^{p+1/2}\|_{H^2(\Omega)}^2 \\
	&= - \sum_{i=1}^n \into \sum_{1 \le j \ne i \le n} K_{ij} u_i^{p+1} u_j^{p+1} \nabla (w_i^{p+1}-  w_j^{p+1}) \cdot \nabla  (w_i^{p+1}- w_0^{p+1/2}) dx \\
	& \quad - \sum_{i=1}^n \into K_{i0} u_i^{p+1} u_0^{p+1} |\nabla  (w_i^{p+1}- w_0^{p+1/2})|^2 dx- \tau \sum_{i=1}^n \| w_i^{p+1}- w_0^{p+1/2}\|_{H^2(\Om)}^2 \\
	& \le -\into (\nabla  { w}^{p+1})^T M({ u}^{p+1}) \nabla { w}^{p+1} dx- \sum_{i=1}^n \into K_{i0} u_i^{p+1} u_0^{p+1} |\nabla  (w_i^{p+1}- w_0^{p+1/2})|^2 dx \\
	& \le - \sum_{i=1}^n \into K_{i0} u_i^{p+1} u_0^{p+1} |\nabla  (w_i^{p+1}- w_0^{p+1/2})|^2 dx,
	\end{split}
	\end{equation*}
where we applied Lemma \ref{rem1} with the vectors $ {w}^{p+1}$ and $u^{p+1} $
and the matrix $M$ given by \eqref{mobility}. Then reasoning again as in Theorem~\ref{thm1} gives
	\[
	\sum_{i=1}^n \int_0^T  \into K_{i0} u_i^{(\tau)} u_0^{(\tau)} |\nabla  (w_i^{(\tau)}- w_0^{(\tau)})|^2 dx dt \le C,
	\]
and hence the conclusion follows.
\end{proof}

We finally point out that the a-priori estimates collected in Lemmas \ref{thm1}-\ref{thm2} allow us to get the following

\begin{lemma}
There exists $C>0$, independent of $\tau>0$, such that 
$$
\int_0^T \left\|\frac{u_i^{(\tau)}(t) - \sigma_\tau u_i^{(\tau)}(t)}{\tau}\right\|_{(H^{2}(\Omega))'}^2 \,dt \leq  C \quad \text{ for all }  1\leq i \leq n. 
$$

\end{lemma}


\begin{proof}
We fix $i \in \{1, \dots, n\}$ and $\phi_i \in H^2(\Om)$. Then, for all $p\in \mathbb{N}$, taking into account the fact that $0 \leq u_j^{p+1} \leq 1$ for all $1\leq j \leq n$ and using Cauchy-Schwarz inequality, we obtain 
	\[
	\begin{split}
	&\l|\frac{1}{\tau} \into (u_i^{p+1}- u_i^p) \phi_i dx \r| \\
	& \le \into \left(\sum_{1 \le j \ne i \le n}  |K_{ij} u_i^{p+1} u_j^{p+1} \nabla(w_i^{p+1}- w_j^{p+1})|+ |K_{i0} u_i^{p+1} u_0^{p+1} \nabla(w_i^{p+1}- w_0^{p+1/2})|\right) |\nabla \phi_i|dx \\
	& \qquad + \tau \langle w_i^{p+1}- w_0^{p+1/2}, \phi_i\rangle_{H^2(\Omega)}\\
	& \le \into \left(\sum_{1 \le j \ne i \le n} K_{ij} \l(u_j^{p+1} |\nabla u_i^{p+1}|+ u_i^{p+1} |\nabla u_j^{p+1}| \r)+ K_{i0} u_i^{p+1} u_0^{p+1} |\nabla w_i^{p+1}- w_0^{p+1/2}|\right)|\nabla \phi_i| dx \\
	& \qquad + \tau  \|w_i^{p+1}- w_0^{p+1/2}\|_{H^2(\Om)} \|\phi_i\|_{H^2(\Om)} \\
	& \le C\left( \sum_{j=1}^n \|u_j^{p+1}\|_{H^1(\Omega)} +\|w_i^{p+1}- w_0^{p+1/2}\|_{H^1(\Omega)}\right) \|\phi_i\|_{H^1(\Om)} + \tau \|w_i^{p+1}- w_0^{p+1/2}\|_{H^2(\Om)} \|\phi_i\|_{H^2(\Om)}.
	\end{split}
	\]
Using the previous estimates proved in this section gives the desired result..
\end{proof}

\section{Passing to the limit as $\tau\to 0$ and proof of Theorem~\ref{thm:existence}}
\label{sect4}

The aim of this section is to identify a weak solution to (\ref{eq:finalsys}) in the sense of Definition~\ref{def1}  as the weak limit  of some extracted subsequence of $(u^{(\tau)})_{\tau >0}$ as $\tau \to 0^+$. 
Passing to the limit  can be done for most terms of the system using either standard arguments in the analysis of cross-difusion systems by the boundedness-by-entropy method (see~\cite{Juengel2015boundedness}), or 
of the Cahn-Hillard model with classical degenerate mobility (see~\cite{EG}). However, some terms appearing in the system require specific arguments, which are new at least up to our knowledge, and which we detail below. Where not differently specified, the limit will be always understood as $\tau \to 0^+$. 

The different estimates collected in Section~\ref{apriori-sect} yield the existence of $\bu = (u_0, \dots, u_n)\in  L^2((0,T); H^2(\Omega)) \times (L^2((0,T); H^1(\Omega)))^{n}$ such that $0\leq u_i \leq 1$ for all 
$0\leq i \leq n$, $u_0 = 1 - \sum\limits_{i=1}^n u_i$, and such that
up to the extraction of a subsequence, 
	\[
	\begin{aligned}
	u_i^{(\tau)} &\mathop{\rightharpoonup}  u_i \quad && \mbox{ weakly in }L^2((0,T); H^1(\Omega)),\\
	\frac{u_i^{(\tau)} - \sigma_\tau u_i^{(\tau)}}{\tau}  & \mathop{\rightharpoonup}  \partial_t u_i && \mbox{ weakly in }L^2((0,T); (H^2(\Omega))'),
	\end{aligned}
	\]
for all $1\leq i \leq n$ and 
	\[
	\begin{aligned}
	u_0^{(\tau)}(t) &\mathop{\rightharpoonup}  u_0 \quad && \mbox{ weakly in }L^2((0,T); H^2(\Omega)),\\
	\frac{u_0^{(\tau)} - \sigma_\tau u_0^{(\tau)}}{\tau}  & \mathop{\rightharpoonup}  \partial_t u_0 && \mbox{ weakly in }L^2((0,T); (H^2(\Omega))').
	\end{aligned}
	\]
Using \cite[Theorem~1]{Dreher2012}, we also obtain that  $ u_i^{(\tau)} \to u_i$  strongly in $L^2((0,T); L^2(\Omega))$ 
and  $ u_0^{(\tau)} \to u_0$  strongly in $L^2((0,T); H^1(\Omega))$ and $L^2((0,T); L^\infty(\Omega))$. This is a consequence of the compact embeddings $H^1(\Omega) \hookrightarrow L^2(\Omega) $, $H^2(\Omega) \hookrightarrow L^\infty(\Omega)$, and $H^2(\Omega) \hookrightarrow H^1(\Omega)$. The uniform bound of $\left(u_i^{(\tau)}\right)_{\tau >0}$ in $L^\infty((0,T); L^\infty(\Omega))$ implies that, up to the extraction of a subsequence 
$$
u_i^{(\tau)}\mathop{\to}  u_i,  \mbox{ strongly in }L^p((0,T); L^p(\Omega)), \quad \forall \, 1\leq p < +\infty, \, 0\leq i \leq n.
$$
Moreover, the uniform bound of $\left(\nabla u_0^{(\tau)}\right)_{\tau >0}$ in $L^\infty((0,T); (L^2(\Omega))^d)$    implies that, up to the extraction of a subsequence, 
$$
\nabla u_0^{(\tau)}\mathop{\to}  \nabla u_0 \quad  \mbox{ strongly in }L^p((0,T); (L^2(\Omega))^d).
$$
Furthermore, up to the extraction of a subsequence,  $\left( \sigma_\tau u_0^{(\tau)}\right)_{\tau >0}$  converges to $u_0$  weakly in $L^2((0,T); H^2(\Omega))$  and strongly in 
$L^2((0,T); H^1(\Omega))$ and $L^2((0,T); L^\infty(\Omega))$.
Finally, Remark \ref{rem2} gives
$$
\tau\left( w_i^{(\tau)} -w_0^{(\tau)} \right) \mathop{\to}  0 \quad \mbox{ strongly in }L^2((0,T); H^2(\Omega)),
$$
for all $1\leq i \leq n$. 

\medskip

Equations \eqref{dis-sys1} and \eqref{eq:w_to_u_weak0} imply that, for all $1\leq i \leq n$,
	\begin{equation}
	\label{weak-form1}
	\begin{split}
	\int_\tau^T \into \frac{u_i^{(\tau)}- \sigma_\tau u_i^{(\tau)}}{\tau} \phi_i dx dt &= -\int_\tau ^T \into \biggl(\sum_{1 \le j \ne i \le n} K_{ij} u_i^{(\tau)} u_j^{(\tau)} \nabla (w_i^{(\tau)}- w_j^{(\tau)})\\
	& \quad \quad+ K_{i0} u_i^{(\tau)} u_0^{(\tau)} \nabla (w_i^{(\tau)}- w_0^{(\tau)}) \biggr) \cdot \nabla \phi_i dx dt \\
	& \quad -\tau \int_\tau^T \langle w_i^{(\tau)}- w_0^{(\tau)}, \phi_i \rangle_{H^2(\Omega)} dt,
	\end{split}
	\end{equation}
for all piecewise constant functions $\phi_i: (0,T) \to H^2(\Omega)$,  with 
\begin{equation}\label{eq:rel00}
w_i^{(\tau)} = \ln u_i^{(\tau)} \quad \mbox{ and } \quad w_0^{(\tau)} = - \eps \Delta u_0^{(\tau)} + \beta (1 - 2 \sigma_\tau u_0^{(\tau)}). 
\end{equation}
Since the set of such $\phi_i$ is dense in $L^2((0,T); H^2(\Omega))$, the weak formulation \eqref{weak-form1} also holds for all $\phi_i \in L^2((0,T); H^2(\Omega))$. Using (\ref{eq:rel00}) then we can rewrite \eqref{weak-form1} equivalently as follows: for all $\phi_i \in L^2((0,T); H^2(\Omega))$, 
	\begin{align}\label{eq:weak_time_discrete}
	\begin{split}
	\int_\tau^T \into \frac{u_i^{(\tau)}- \sigma_\tau u_i^{(\tau)}}{\tau} \phi_i dx dt &= -\int_\tau ^T \into \sum_{1 \le j \ne i \le n} K_{ij}\left( u_j^{(\tau)}  \nabla u_i^{(\tau)}  - u_i^{(\tau)}  \nabla u_j^{(\tau)} \right) \cdot \nabla \phi_i dx dt\\
	& \quad \quad - \int_\tau^T \into  K_{i0} u_0^{(\tau)} \nabla u_i^{(\tau)} \cdot \nabla \phi_i  + 	\int_\tau^T \into   K_{i0} u_0^{(\tau)} u_i^{(\tau)}   \nabla w_0^{(\tau)} \cdot \nabla \phi_i dx dt\\
	& \quad -\tau  \int_\tau^T \langle w_i^{(\tau)}- w_0^{(\tau)}, \phi_i \rangle_{H^2(\Omega)} dt,
	\end{split}
	\end{align}
\medskip
The different convergences identified above enable to easily identify the limit as $\tau \to 0^+$, following standard arguments in the study of cross-diffusion systems (see for instance~\cite{Juengel2015boundedness}). More precisely, for all 
$1\leq i\neq j \leq n$ it holds that
	\begin{align}\label{eq:easy_limits}
	\begin{split}
	 \int_\tau^T \into \frac{u_i^{(\tau)}- \sigma_\tau u_i^{(\tau)}}{\tau} \phi_i dx dt & \mathop{\to}  \int_0^T \langle \partial_t u_i ,\phi_i \rangle_{(H^2(\Omega))', H^2(\Omega)} dt,\\
	 \int_\tau ^T \into \sum_{1 \le j \ne i \le n} K_{ij}\left( u_j^{(\tau)}  \nabla u_i^{(\tau)}  - u_i^{(\tau)}  \nabla u_j^{(\tau)} \right) \cdot \nabla \phi_i dx dt & \mathop{\to} \int_0^T \into \sum_{1 \le j \ne i \le n} K_{ij}\left( u_j  \nabla u_i  - u_i  \nabla u_j \right) \cdot \nabla \phi_i dx dt,\\
	 \int_\tau^T \into  K_{i0} u_0^{(\tau)} \nabla u_i^{(\tau)} \cdot \nabla \phi_i dx dt & \mathop{\to}  \int_0^T \into  K_{i0} u_0 \nabla u_i\cdot \nabla \phi_i dx dt,\\
	 \tau  \int_\tau^T \langle w_i^{(\tau)}- w_0^{(\tau)}, \phi_i \rangle_{H^2(\Omega)} dt  & \mathop{\to}  0.
	 \end{split}
	\end{align}
Of course, all these convergences hold up to the extraction of subsequences. Passing to the limit in the term 
$$
\displaystyle \int_\tau^T \into   K_{i0} u_0^{(\tau)} u_i^{(\tau)}   \nabla w_0^{(\tau)} \cdot \nabla \phi_i dx dt
$$ 
requires specific arguments which is the object of the following lemma. We set
$w_0:= - \eps \Delta u_0 + \beta (1-2u_0)$ and point out that the convergences stated above imply that $ w_0^{(\tau)} \rightharpoonup w_0$ weakly in $L^2((0,T); L^2(\Omega))$ as $\tau \to 0^+$.

\begin{lemma}\label{lem:lastlim}
	There exists $J\in L^2((0,T); (L^2(\Omega))^d)$ which satisfies $J = (1-u_0)u_0 \nabla w_0$ in the weak sense, i.e.,
	\begin{equation*} 
	\int_0^T\int_\Omega J \cdot \bta dx dt=  - \int_0^T \int_\Omega w_0  \divergenz (u_0 (1-u_0) \bta) dx dt, 
	\end{equation*}
	for all $\bta \in L^2((0,T); (H^1(\Omega))^d) \cap L^\infty((0,T)\times \Omega; \mathbb{R}^d)$ with $\bta \cdot {\bf n} = 0$ on $\partial \Omega \times (0,T)$ and such that, up to the extraction of a subsequence,
	\begin{equation}\label{eq:term0}
	(1-u_0^{(\tau)})u_0^{(\tau)} \nabla w_0^{(\tau)} \mathop{\rightharpoonup}  J \quad \mbox{ weakly in }  L^2((0,T); (L^2(\Omega))^d)
	\end{equation}
	and  
	\begin{equation}\label{eq:termi}
	u_i^{(\tau)} u_0^{(\tau)} \nabla w_0^{(\tau)} \mathop{\rightharpoonup} \frac{u_i}{1-u_0} J \quad \mbox{ weakly in }  L^2((0,T); (L^2(\Omega))^d)
	\end{equation}
for all $1\leq i \leq n$.
\end{lemma}
\begin{remark}
 The weak limit (\ref{eq:term0}) can be obtained using classical arguments for the standard Cahn-Hilliard system (see~\cite{EG} for instance). We recall them in the proof below for the sake of completeness. 
 However, obtaining (\ref{eq:termi}) is not standard, at least up to our knowledge, and 
 the arguments which yield 
 to this convergence are detailed in the proof below. Let us mention here that one difficulty in the analysis is that the sequence 
 $\left( \frac{u_i^{(\tau)}}{1-u_0^{(\tau)}}\right)_{\tau >0}$ does not converge a priori in any sense to $\frac{u_i}{1-u_0}$ if $1-u_0 = 0$ in 
 some parts of the domain $\Omega$. 
\end{remark}

\begin{proof}
	From \eqref{dis.d} we know that 
	$$
	\int_0^T \int_\Omega |u_0^{(\tau)} (1-u_0^{(\tau)}) \nabla w_0^{(\tau)}|^2 dx dt \leq \int_0^T \int_\Omega u_0^{(\tau)} (1-u_0^{(\tau)}) |\nabla w_0^{(\tau)}|^2 dx dt \leq C,
	$$
for every $\tau> 0$. Then,
	up to the extraction of a subsequence, there exists $J\in L^2((0,T); (L^2(\Omega))^d)$ such that 
	\[
	u_0^{(\tau)} (1-u_0^{(\tau)}) \nabla w_0^{(\tau)} \rightharpoonup J \quad \text{ weakly in }  L^2((0,T); (L^2(\Omega))^d).
	\]
 Let us  now take $\bta \in L^2((0,T); (H^1(\Omega))^d) \cap L^\infty((0,T)\times \Omega; \mathbb{R}^d)$ which fulfills $\bta \cdot {\bf n} = 0$ on $\partial \Omega \times (0,T)$. Integrating by parts gives
	\begin{equation}
	\label{int-parts}
	\begin{split}
	& \int_0^T \int_\Omega u_0^{(\tau)}(1-u_0^{(\tau)}) \nabla w^{(\tau)}_0 \bta dx dt \\
	&= - \int_0^T \int_\Omega (1-2u_0^{(\tau)})w^{(\tau)}_0  \nabla u^{(\tau)}_0 \cdot \bta dx dt - \int_0^T \int_\Omega u_0^{(\tau)}
	(1-u_0^{(\tau)}) w^{(\tau)}_0 \divergenz \bta dx dt. 
	\end{split}
	\end{equation}
Since  $w^{(\tau)}_0 \rightharpoonup w_0$ weakly in $L^2((0,T); L^2(\Omega))$, the strong convergence of $\nabla u^{(\tau)}$ together with the fact that $u_0^{(\tau)}$ converges a.e. and is uniformly bounded implies 
$$
 (1-2u_0^{(\tau)})w^{(\tau)}_0  \nabla u^{(\tau)}_0 \rightharpoonup  (1-2u_0^{(\tau)})w^{(\tau)}_0  \nabla u^{(\tau)}_0,\quad \text{weakly in } L^1((0,T);L^1(\Omega)^d)
$$
and enables us to pass to the limit in the first term on the right hand side of \eqref{int-parts}. For the second term we argue again using the a.e. convergence of $u_0^{(\tau)}$ and thus obtain (\ref{eq:term0}).  
	
	\medskip
	
	Let us now prove  the weak convergence (\ref{eq:termi}). We know 
	that, up to the extraction of a subsequence, $ u_0^{(\tau)} \to u_0$ strongly in $L^2((0,T); L^\infty(\Omega))$. 
	This implies that for almost all $t\in (0,T)$
	$$
	\|u_0^{(\tau)}(t) - u_0(t)\|_{L^\infty(\Omega)} = C_{\tau}(t)
	$$
	where $C_\tau$ satisfies $\displaystyle \int_0^T C_\tau(t) \,dt \mathop{\to}  0$ (using a Cauchy-Schwarz inequality). In particular, for all $\delta >0$, denoting by
	$$
	E^{\delta, \tau}:= \{ t\in [0,T], \; \|u_0^{(\tau)}(t) - u_0(t)\|_{L^\infty(\Omega)} > \delta \},
	$$
	it holds that the Lebesgue measure of the set $E^{\delta, \tau}$ goes to $0$ as $\tau$ goes to $0$. We also consider the complementary of $E^{\delta, \tau}$, i.e., the set 
	$$
	E^{\delta, \tau, c}:= \{ t\in [0,T], \; \|u_0^{(\tau)}(t) - u_0(t)\|_{L^\infty(\Omega)} \leq \delta \}.
	$$
	Let now $\epsilon >0$ and let us introduce the set 
	\begin{align*}
	M^\epsilon(t):= \left\{ x\in \Omega, 1-u_0(t,x) \geq \epsilon\right\},
	\end{align*}
	together with its complementary
	\begin{align*}
		M(t)^{\epsilon,c}:= \left\{ x\in \Omega, 1- u_0(t,x) < \epsilon\right\}.
	\end{align*}
	For all $t>0$ we can write
	\begin{equation}\label{eq:identity}
	\begin{split}
	& u_0^{(\tau)}(t) u_i^{(\tau)}(t) \nabla w_0^{(\tau)}(t)  \\
	&= \chi_{M^\epsilon(t)}u_0^{(\tau)}(t) u_i^{(\tau)}(t) \nabla w_0^{(\tau)}(t) + \chi_{M^{\epsilon,c}(t)}u_0^{(\tau)}(t) u_i^{(\tau)}(t) \nabla w_0^{(\tau)}(t).
	\end{split}
	\end{equation}
	Let us consider both terms separately. 
	On the one hand, it holds that $u_i^{(\tau)} \leq 1 - u_0^{(\tau)}$. Thus, we have 
	$$
	\int_0^T \left\|\chi_{M^{\epsilon, c}(t)}u_0^{(\tau)}(t) u_i^{(\tau)}(t) \nabla w_0^{(\tau)}(t)\right\|_{L^2(\Omega)}^2 \,dt \leq \int_0^T \left\|u_0^{(\tau)} (1 -u_0^{(\tau)}) \nabla w_0^{(\tau)}\right\|_{L^2(\Omega)}^2 dt \leq C, 
	$$
	for some constant $C>0$ independent of $\tau >0$. Hence, if we consider the function $h^{\epsilon, (\tau)}: (0, T) \times \Om \to \R$ such that $h^{\epsilon, (\tau)}(t, x)= \chi_{M^{\epsilon, c}(t)}u_0^{(\tau)}(t,x) u_i^{(\tau)}(t,x) \nabla w_0^{(\tau)}(t,x)$, it follows that 
	there exists a function $h^\epsilon \in L^2((0,T); (L^2(\Omega))^d)$ such that
	\begin{equation}\label{eq:wlimL2}
	h^{\epsilon, (\tau)} \mathop{\rightharpoonup}  h^\epsilon \quad \mbox{ weakly in }  L^2((0,T); (L^2(\Omega))^d). 
	\end{equation}
	Besides, since $\|h^\epsilon\|_{L^2((0,T); (L^2(\Omega))^d)} \leq C$ for all $\epsilon >0$, there exists $h \in L^2((0,T); (L^2(\Omega))^d)$ such that, up to the extraction of a subsequence, 
	$$
	h^\epsilon \mathop{\rightharpoonup} h \quad \mbox{ weakly in }  L^2((0,T); (L^2(\Omega))^d) \quad \text{ as } \epsilon\to 0^+.
	$$
Let us now show that necessarily $h = 0$. Equation~\eqref{eq:wlimL2} implies that 
	$$
	h^{\epsilon, (\tau)}  \mathop{\rightharpoonup} h^\epsilon \quad \mbox{ weakly in }  L^1((0,T); (L^2(\Omega))^d), 
	$$
	and that 
	$$
	\|h^\epsilon\|_{L^1((0,T); L^2(\Omega))} \leq \mathop{\liminf}_{\tau \to 0} \left\| h^{\epsilon, (\tau)} \right\|_{L^1((0,T); L^2(\Omega))}.
	$$
	To prove a bound on the right hand side we exploit the fact that $u_i^{(\tau)} \le 1-u_0^{(\tau)}$ and $u_0^{(\tau)} \le 1$ to estimate
	\begin{align*}
	& \left\| h^{\epsilon, (\tau)}  \right\|_{L^1(0,T; L^2(\Omega))} = \int_0^T\left\|\chi_{M^{\epsilon, c}(t)}u_0^{(\tau)}(t) u_i^{(\tau)}(t) \nabla w_0^{(\tau)}(t)\right\|_{L^2(\Omega)}\,dt\\
	& \le \int_0^T\left\|\chi_{M^{\epsilon, c}(t)}u_0^{(\tau)}(t) (1-u_0^{(\tau)}(t)) \nabla w_0^{(\tau)}(t)\right\|_{L^2(\Omega)}\,dt\\
	& \leq \int_0^T \left\|\chi_{M^{\epsilon, c}(t)}\sqrt{1-u_0^{(\tau)}(t)} \right\|_{L^\infty(\Omega)} \left\|\sqrt{1- u_0^{(\tau)}(t)} \sqrt{u_0^{(\tau)}(t)}\nabla w_0^{(\tau)}(t)\right\|_{L^2(\Omega)}\,dt\\
	& \leq \left(\int_0^T \left\|\chi_{M^{\epsilon, c}(t)}\sqrt{1-u_0^{(\tau)}(t)} \right\|_{L^\infty(\Omega)}^2\,dt\right)^{1/2}\left( \int_0^T \left\|\sqrt{1- u_0^{(\tau)}(t)} \sqrt{u_0^{(\tau)}(t)}\nabla w_0^{(\tau)}(t)\right\|^2_{L^2(\Omega)}\,dt\right)^{1/2}\\
	& \leq C \left(\int_0^T \left\|\chi_{M^{\epsilon, c}(t)}\sqrt{1-u_0^{(\tau)}(t)} \right\|_{L^\infty(\Omega)}^2\,dt\right)^{1/2},\\
	\end{align*}
	for some constant $C>0$ independent of $\tau$. We further estimate the integrand of the last term of the inequality above as 
	\begin{align*}
	&  \left\|\chi_{M^{\epsilon, c}(t)}\sqrt{1-u_0^{(\tau)}(t)} \right\|_{L^\infty(\Omega)} \\
	& \leq \left\|\chi_{M^{\epsilon, c}(t)}\left(\sqrt{1-u_0(t)} + \sqrt{|u_0(t) - u_0^{(\tau)}(t)|}\right) \right\|_{L^\infty(\Omega)}\\
	& \leq \left\|\chi_{M^{\epsilon, c}(t)}\sqrt{1-u_0(t)}\right\|_{L^\infty(\Omega)} + \left\|\chi_{M^{\epsilon, c}(t)}\sqrt{|u_0(t) - u_0^{(\tau)}(t)|} \right\|_{L^\infty(\Omega)}\\
	& \leq \sqrt{\epsilon} + \chi_{E^{\epsilon, \tau}}(t)\left\|\chi_{M^{\epsilon, c}(t)}\sqrt{|u_0(t) - u_0^{(\tau)}(t)|}  \right\|_{L^\infty(\Omega)} + \chi_{E^{\epsilon, \tau, c}}(t) \left\|\chi_{M^{\epsilon, c}(t)}\sqrt{|u_0(t) - u_0^{(\tau)}(t)|} \right\|_{L^\infty(\Omega)} \\
	& \leq \sqrt{\epsilon} + 2 \chi_{E^{\epsilon, \tau}}(t) +  \sqrt{\epsilon} \leq 2 \sqrt{\epsilon} + 2 \chi_{E^{\epsilon, \tau}}(t).\\
	\end{align*}
	This gives
	\begin{align*}
	\int_0^T \left\|\chi_{M^{\epsilon, c}(t)}\sqrt{1-u_0^{(\tau)}(t)} \right\|_{L^\infty(\Omega)}^2\,dt & \leq 4 T \epsilon + \int_0^T \chi_{E^{\epsilon, \tau}}(t) (4 + 4 \sqrt{\epsilon})\,dt, \\
	& \leq  4 T \epsilon + (4 + 4 \sqrt{\epsilon}) |E^{\epsilon, \tau}|. \\
	\end{align*}
	Thus, since $\displaystyle |E^{\epsilon, \tau}| \mathop{\to}  0$ as $\tau \to 0^+$, we  obtain that
	$$
	\mathop{\liminf}_{\tau \to 0} \int_0^T \left\|\chi_{M^{\epsilon, c}(t)}\sqrt{1-u_0^{(\tau)}(t)} \right\|_{L^\infty(\Omega)}^2\,dt \leq 4 T \epsilon.
	$$
This implies that
	$$
	\mathop{\liminf}_{\tau\to 0} \int_0^T\left\|\chi_{M^{\epsilon, c}(t)}u_0^{(\tau)}(t) u_i^{(\tau)}(t) \nabla w_0^{(\tau)}(t)\right\|_{L^2(\Omega)}\,dt =\mathop{\liminf}_{\tau \to 0} \left\| h^{\epsilon, (\tau)}  \right\|_{L^1((0,T); L^2(\Omega))} \leq 2C\sqrt{T \epsilon}.
	$$
	As a consequence, 
	$$
	\|h^\epsilon\|_{L^1((0,T); L^2(\Omega))} \leq  2C \sqrt{T \epsilon},
	$$
that is, $  \|h^\epsilon\|_{L^1((0,T); L^2(\Omega))} \to 0$ as $\epsilon \to 0^+$. Moreover, since $ h^\epsilon \rightharpoonup h$ weakly in $L^2((0,T); L^2(\Omega))$, then the weak convergence holds also in   $L^1((0,T); L^2(\Omega))$. This implies that $h = 0$.

	\medskip

	Let us now consider the second term in (\ref{eq:identity}). Let  $g^{\epsilon, (\tau)}: (0,T)\times \Omega \to \R$ be defined by $g^{\epsilon, (\tau)} (t, x) = \chi_{M^{\epsilon, c}(t)}u_0^{(\tau)}(t,x) u_i^{(\tau)}(t,x) \nabla w_0^{(\tau)}(t,x)$. Since
	$$
	\left\|g^{\epsilon, (\tau)}\right\|_{L^2((0,T), L^2(\Omega))} \leq \left\|u_0^{(\tau)} (1 -u_0^{(\tau)}) \nabla w_0^{(\tau)}\right\|_{L^2((0,T); L^2(\Omega))} \leq C,
	$$
then  there exists $g^\epsilon \in L^2((0,T); L^2(\Omega))$ such that, up to the extraction of a subsequence, $ g^{\epsilon, (\tau)} \to  g^\epsilon $ weakly  as $\tau \to 0^+$.  
	Let us prove that 
		\[
		g^\epsilon(t,x) = \chi_{M^\epsilon(t)}(x)\frac{u_i(t,x)}{1-u_0(t,x)} J(t,x) \quad \text{ for almost all } (t,x)\in (0,T) \times \Omega.
		\] 
	On the one hand, we have 
	\[
	 \chi_{M^\epsilon(t)}(x)\frac{u_i^{(\tau)}(t,x)}{1-u_0^{(\tau)}(t,x)} \to \chi_{M^\epsilon(t)}(x)\frac{u_i(t,x)}{1-u_0(t,x)} \quad \text{ for almost all } (t, x)\in (0,T)\times \Omega.
	\]
Besides, $\chi_{M^\epsilon}(x)\frac{u_i^{(\tau)}(t)}{1-u_0^{(\tau)}(t)} \leq 1$ for almost all $t\in (0,T)$ so that 
Lebesgue's dominated convergence theorem implies that, up to the extraction of a subsequence,
	\[
	\chi_{M^\epsilon}(x)\frac{u_i^{(\tau)}(t)}{1-u_0^{(\tau)}(t)} \to f_i^{\epsilon} \quad \text{strongly},
	\]
	in any $L^p((0,T); L^p(\Omega))$ for all $p>1$, in particular in   $L^2((0,T); L^2(\Omega))$.
This, together with the fact that $(1-u_0^{(\tau)}) u_0^{(\tau)}\nabla w_0^{(\tau)} \to J$ weakly in $L^2((0,T); (L^2(\Omega))^d)$ yields that
	\[
	 g^{\epsilon, (\tau)} \to f_i^\epsilon J
	 \]
 in the sense of distribution. Hence, by uniqueness of the limit, we have
 	\[
	g^\epsilon = f_i^\epsilon J, 
	\]
which was the desired result. Thus, in the distributional sense,  
	$$
	u_0^{(\tau)} u_i^{(\tau)} \nabla w_0^{(\tau)} \mathop{\to} f_i^{\epsilon}J + h^\epsilon.
	$$
	Now, since $1-u_0^{(\tau)} \to 1-u_0$ strongly in $L^\infty(\Omega)$ and $u_i^{(\tau)} \to u_i$ almost everywhere it holds that $ f_i^{\epsilon} \to \kappa_i(t,x)$ almost everywhere, as $\epsilon \to 0^+$, being $\kappa_i(t,x)$ as defined in \eqref{eq:def_kappa_i}.
	 Thus,  the Lebesgue dominated convergence theorem gives  $ f_i^{\epsilon} \to \kappa_i $ strongly in $L^2((0,T); L^2(\Omega))$ as $\epsilon \to 0^+$. Therefore, 
	$$
	f_i^{\epsilon}J \mathop{\to}  \kappa_i J,
	$$
	so that finally, using the fact that $ h^\epsilon \rightharpoonup 0$ weakly in $L^2((0,T); (L^2(\Omega))^d)$, we obtain that
	\[
	u_0^{(\tau)} u_i^{(\tau)} \nabla w_0^{(\tau)} \rightharpoonup \kappa_i J \quad \text{weakly in } L^2((0,T); (L^2(\Omega))^d)
	\]
 as $\tau \to 0^+$, which was the desired result.
\end{proof}
We are now in a position to complete the proof of our main theorem.
\begin{proof}[Proof of Theorem~\ref{thm:existence}] We pass to the limit $\tau \to 0^+$ in \eqref{eq:weak_time_discrete} using \eqref{eq:easy_limits} and Lemma~\ref{lem:lastlim} which enables us to identify the limit
$$
\int_\tau^T \into   K_{i0} u_0^{(\tau)} u_i^{(\tau)}   \nabla w_0^{(\tau)} \cdot \nabla \phi_i dx dt \mathop{\to}  \int_0^T \into   K_{i0} \kappa_i J \cdot \nabla \phi_i dx dt,
$$
for all $1\leq i \leq n$. Thus, for all $1\leq i \leq n$ and all $\phi_i \in L^2((0,T); H^2(\Omega))$
	\begin{align*}
	\int_0^T \langle \partial_t u_i , \phi_i \rangle_{H^2(\Omega)', H^2(\Omega)} dt &= -\int_0^T \into \sum_{1 \le j \ne i \le n} K_{ij}\left( u_j  \nabla u_i  - u_i  \nabla u_j \right) \cdot \nabla \phi_i dx dt \\
	& \quad \quad - \int_0^T \into  K_{i0} u_0\nabla u_i \cdot \nabla \phi_i dx dt + \int_0^T \into   K_{i0} \kappa_iJ \cdot \nabla \phi_i dx dt. \\
\end{align*}
From the obtained weak formulation, it is clear that $\partial_t u_i \in L^2((0,T); (H^1(\Omega))')$ and that, by density, we can extend the above formulation to all $\phi_i \in L^2((0,T); H^1(\Omega))$ as follows: 
	\begin{align*}
	\int_0^T \langle \partial_t u_i , \phi_i \rangle_{H^1(\Omega)', H^1(\Omega)} dt &= -\int_0^T \into \sum_{1 \le j \ne i \le n} K_{ij}\left( u_j  \nabla u_i  - u_i  \nabla u_j \right) \cdot \nabla \phi_i dx dt \\
	& \quad \quad - \int_0^T \into  K_{i0} u_0\nabla u_i \cdot \nabla \phi_i dx dt + 	\int_0^T \into   K_{i0} \kappa_i J \cdot \nabla \phi_i dx dt. \\
\end{align*}
Lastly, we obtain that, necessarily, $u_i(0,\cdot) = u_i^0$ using similar arguments as in~\cite{Juengel2015boundedness}. 
Hence $((u_i)_{0\leq i \leq n}, J)$ is a weak solution of system (\ref{eq:finalsys}) in the sense of Definition~\ref{def1}, which concludes the proof of Theorem~\ref{thm:existence}. 
\end{proof}

\subsection*{Acknowledgements}\mbox{}
VE acknowledges support from the ANR JCJC project COMODO (ANR-19-CE46-0002) and from the PHC PROCOPE project Number 42632VA. GM and JFP were supported by the DAAD via the PPP Grant No. 57447206.

\bibliographystyle{plain}
\bibliography{CrossCH}

\begin{thebibliography}{10}

\bibitem{Alves2017_systemlognonlin}
C.~Alves, A.~Moussaoui, and L.~Tavares.
\newblock {A}n elliptic system with logarithmic nonlinearity.
\newblock {\em Adv. Nonlinear Anal.}, 8(1):928--945, 2017.

\bibitem{BE2018freeboundary}
A.~Bakhta and V.~Ehrlacher.
\newblock {C}ross-diffusion systems with non-zero flux and moving boundary
  conditions.
\newblock {\em ESAIM Math. Model. Numer. Anal.}, 52(4):1385--1415, 2018.

\bibitem{Berendsen2019}
J.~Berendsen, M.~Burger, V.~Ehrlacher, and J.-F. Pietschmann.
\newblock {U}niqueness of strong solutions and weak--strong stability in a
  system of cross-diffusion equations.
\newblock {\em J. Evol. Equ.}, Sep 2019.

\bibitem{Berendsen2017nonlocal}
J.~Berendsen, M.~Burger, and J.-F. Pietschmann.
\newblock {O}n a cross-diffusion model for multiple species with nonlocal
  interaction and size exclusion.
\newblock {\em Nonlinear Anal.}, 159:10--39, 2017.
\newblock Advances in Reaction-Cross-Diffusion Systems.

\bibitem{boyer2014hierarchy}
F.~Boyer and S.~Minjeaud.
\newblock Hierarchy of consistent n-component {C}ahn-{H}illiard systems.
\newblock {\em Math. Mod. Meth. Appl. S.}, 24(14):2885--2928, 2014.

\bibitem{BDPS2010}
M.~Burger, M.~Di Francesco, J.-F. Pietschmann, and B.~Schlake.
\newblock {{N}onlinear {C}ross {D}iffusion with {S}ize {E}xclusion}.
\newblock {\em SIAM J. Math. Anal.}, 42(6):2842--2871, 2010.

\bibitem{Burger2016}
M.~Burger, S.~Hittmeir, H.~Ranetbauer, and M.-T. Wolfram.
\newblock Lane formation by side-stepping.
\newblock {\em SIAM Journal on Mathematical Analysis}, 48(2):981--1005, 2016.

\bibitem{Caffarelli1995}
L.A. Caffarelli and N.E. Muler.
\newblock {A}n ${L}^\infty$ bound for solutions of the {C}ahn-{H}illiard
  equation.
\newblock {\em Arch. Ration. Mech. Anal.}, 133(2):129--144, Dec 1995.

\bibitem{CahnHilliard1958}
J.W. Cahn and J.E. Hilliard.
\newblock {F}ree {E}nergy of a {N}onuniform {S}ystem. {I}. {I}nterfacial {F}ree
  {E}nergy.
\newblock {\em J. Chem. Phys.}, 28(2):258--267, 1958.

\bibitem{cances2020convergent}
Cl{\'e}ment Canc{\`e}s and Beno{\^\i}t Gaudeul.
\newblock A convergent entropy diminishing finite volume scheme for a
  cross-diffusion system.
\newblock {\em arXiv preprint arXiv:2001.11222}, 2020.

\bibitem{Chen2004}
L.~Chen and A.~J{\"u}ngel.
\newblock {A}nalysis of a multidimensional parabolic population model with
  strong cross-diffusion.
\newblock {\em SIAM J. Math. Anal.}, 36(1):301--322 (electronic), 2004.

\bibitem{Chen2006}
L.~Chen and A.~J{\"u}ngel.
\newblock {A}nalysis of a parabolic cross-diffusion population model without
  self-diffusion.
\newblock {\em J. Differential Equations}, 224(1):39--59, 2006.

\bibitem{Dai2017_degenerateginzburg}
M.~Dai, E.~Feireisl, E.~Rocca, G.~Schimperna, and M.~E. Schonbek.
\newblock {A}nalysis of a diffuse interface model of multispecies tumor growth.
\newblock {\em Nonlinearity}, 30(4):1639--1658, mar 2017.

\bibitem{Dai2016}
Shibin Dai and Qiang Du.
\newblock Weak solutions for the cahn--hilliard equation with degenerate
  mobility.
\newblock {\em Archive for Rational Mechanics and Analysis}, 219(3):1161--1184,
  Mar 2016.

\bibitem{Dreher2012}
M.~Dreher and A.~J{\"u}ngel.
\newblock Compact families of piecewise constant functions in ${L}^p(0,t;{B})$.
\newblock {\em Nonlinear Analysis: Theory, Methods \& Applications}, 75(6):3072
  -- 3077, 2012.

\bibitem{EG}
C.M. Elliott and H.~Garcke.
\newblock On the {C}ahn-{H}illiard equation with degenerate mobility.
\newblock {\em SIAM J. Math. Anal.}, 27(2):404--423, 1996.

\bibitem{Elliot1997}
C.M. Elliott and H.~Garcke.
\newblock Diffusional phase transitions in multicomponent systems with a
  concentration dependent mobility matrix.
\newblock {\em Phys. D}, 109(3):242--256, 1997.

\bibitem{elliott1991generalized}
C.M. Elliott and S.~Luckhaus.
\newblock A generalized equation for phase separation of a multi-component
  mixture with interfacial free energy.
\newblock {\em preprint SFB}, 256:195, 1991.

\bibitem{Elliott1986}
C.M. Elliott and Z.~Songmu.
\newblock {O}n the {C}ahn-{H}illiard equation.
\newblock {\em Arch. Ration. Mech. Anal.}, 96(4):339--357, Dec 1986.

\bibitem{eyre1998unconditionally}
David~J Eyre.
\newblock Unconditionally gradient stable time marching the cahn-hilliard
  equation.
\newblock In {\em Materials Research Society Symposium Proceedings}, volume
  529, pages 39--46. Materials Research Society, 1998.

\bibitem{Garcke2018_multiphasetumour}
H.~Garcke, K.F. Lam, R.~N{\"u}rnberg, and E.~Sitka.
\newblock {A} multiphase {C}ahn-{H}illiard-{D}arcy model for tumour growth with
  necrosis.
\newblock {\em Math. Mod. Meth. Appl. S.}, 28(03):525--577, 2018.

\bibitem{Gelantalis2017}
M.~Gelantalis, A.~Wagner, and M.G. Westdickenberg.
\newblock {E}xistence and properties of certain critical points of the
  {C}ahn-{H}illiard energy.
\newblock {\em Indiana Univ. Math. J.}, 66:1827--1877, 2017.

\bibitem{Juengel2015boundedness}
A.~J{\"u}ngel.
\newblock {T}he boundedness-by-entropy method for cross-diffusion systems.
\newblock {\em Nonlinearity}, 28(6):1963, 2015.

\bibitem{klinkert2015comprehension}
Torben Klinkert.
\newblock {\em Comprehension and optimisation of the co-evaporation deposition
  of Cu (In, Ga) Se2 absorber layers for very high efficiency thin film solar
  cells}.
\newblock PhD thesis, Universit{\'e} Pierre et Marie Curie-Paris VI, 2015.

\bibitem{Kfner1996InvariantRF}
K.H.W. K{\"u}fner.
\newblock {I}nvariant regions for quasilinear reaction-diffusion systems and
  applications to a two population model.
\newblock {\em NoDEA-Nonlinear Diff.}, 3:421--444, 1996.

\bibitem{Lepoutre2012}
T.~Lepoutre, M.~Pierre, and G.~Rolland.
\newblock {G}lobal {W}ell-{P}osedness of a {C}onservative {R}elaxed {C}ross
  {D}iffusion {S}ystem.
\newblock {\em SIAM J. Math. Anal.}, 44(3):1674--1693, 2012.

\bibitem{Miranville2019}
A.~Miranville.
\newblock {\em {T}he {C}ahn-{H}illiard {E}quation: {R}ecent {A}dvances and
  {A}pplications}.
\newblock Society for Industrial and Applied Mathematics, Philadelphia, PA,
  2019.

\bibitem{Montenegro2009_lognonline}
M.~Montenegro and O.~Santana [de Queiroz].
\newblock {E}xistence and regularity to an elliptic equation with logarithmic
  nonlinearity.
\newblock {\em J. Differential Equations}, 246(2):482--511, 2009.

\bibitem{novick2008cahn}
A.~Novick-Cohen.
\newblock {T}he {C}ahn-{H}illiard equation.
\newblock {\em Handbook of differential equations: evolutionary equations},
  4:201--228, 2008.

\bibitem{Painter2009}
K.~J. Painter.
\newblock Continuous models for cell migration in tissues and applications to
  cell sorting via differential chemotaxis.
\newblock {\em Bulletin of Mathematical Biology}, 71(5):1117--1147, 2009.

\bibitem{Painter2002}
K.J. Painter and T.~Hillen.
\newblock Volume-filling and quorum-sensing in models for chemosensitive
  movement.
\newblock {\em Can. Appl. Math. Q.}, 10(4):501 -- 543, 2002.

\bibitem{Wenisch2016}
R.~Wenisch, R.~Hübner, F.~Munnik, S.~Melkhanova, S.~Gemming, G.~Abrasonis, and
  M.~Krause.
\newblock Nickel-enhanced graphitic ordering of carbon ad-atoms during physical
  vapor deposition.
\newblock {\em Carbon}, 100:656 -- 663, 2016.

\bibitem{Wu2017}
S.~Wu and J.~Xu.
\newblock {M}ultiphase {A}llen-{C}ahn and {C}ahn-{H}illiard models and their
  discretizations with the effect of pairwise surface tensions.
\newblock {\em J. Comput. Phys.}, 343:10--32, 2017.

\bibitem{Zamponi2017_volumefilling}
N.~Zamponi and A.~J{\"u}ngel.
\newblock {A}nalysis of degenerate cross-diffusion population models with
  volume filling.
\newblock {\em Ann. Inst. H. Poincar\'e Anal. Non Lin\'eare C}, 34(1):1--29,
  2017.

\bibitem{Zamponi2017_corrigendum}
N.~Zamponi and A.~J{\"u}ngel.
\newblock {C}orrigendum to {A}nalysis of degenerate cross-diffusion population
  models with volume filling” [ann. inst. h. poincar\'e 34 (1) (2017) 1--29].
\newblock {\em Ann. Inst. H. Poincar\'e Anal. Non Lin\'eare C}, 34(3):789--792,
  2017.

\end{thebibliography}

\end{document}